\documentclass{amsart}
\usepackage{mathrsfs}
\usepackage{amssymb}
\usepackage{amsmath,amscd}
\usepackage{amsfonts,amssymb}
\usepackage{latexsym}
\usepackage{amsbsy}
\usepackage{graphicx}
\usepackage{overpic}
\usepackage{mathdots}
\usepackage[all]{xy}
\usepackage{cite}
\usepackage{tikz-cd}
\newtheorem{thm}{Theorem}[section]
\newtheorem{theo}[thm]{Theorem}

\newtheorem{lem}[thm]{Lemma}
\newtheorem{prop}[thm]{Proposition}
\newtheorem{coro}[thm]{Corollary}
\newtheorem{exam}{Example}[section]
\newtheorem{examp}[exam]{Example}
\newtheorem*{theo*}{Theorem}

\numberwithin{equation}{section}

\begin{document}
\title{Symplectic conditions on Grassmannian, flag, and Schubert varieties}
\author{Jiajun Xu}
\address{School of Mathematical Sciences, Shanghai Jiao Tong University, Shanghai 200240, China}
\email{s1gh1995@sjtu.edu.cn}

\author{Guanglian Zhang}
\address{School of Mathematical Sciences, Shanghai Jiao Tong University, Shanghai 200240, China}
\email{g.l.zhang@sjtu.edu.cn}
\thanks{*Corresponding author: Guanglian Zhang, g.l.zhang@sjtu.edu.cn. \\MSC: 14M15, 14L30, 15A15}


\begin{abstract}
In this paper, a description of the set-theoretical defining equations of symplectic (type C) Grassmannian/flag/Schubert varieties in corresponding (type A) algebraic varieties is given as linear polynomials in Pl$\ddot{u}$cker coordinates, and it is proved that such equations generate the defining ideal of variety of type C in those of type A. As applications of this result, the number of local equations required to obtain the Schubert variety of type C from the Schubert variety of type A is computed, and  further geometric properties of the Schubert variety of type C are given in the aspect of complete intersections. Finally, the smoothness of Schubert variety in the non-minuscule or cominuscule Grassmannian of type C is discussed, filling gaps in the study of algebraic varieties of the same type.

{\bf\emph{ Keywords: Grassmannian variety; generalized flag variety; Schubert variety; Pl$\ddot{u}$cker embedding; complete intersection. }\rm}
\end{abstract}

\maketitle
\section{Introduction}
Grassmannian and flag varieties, which stem from linear algebra, are important study objects in the interplay of algebraic geometry, representation theory, and combinatorics. The symplectic Grassmannian and flag variety have also attracted considerable interest from researchers (eg.\cite{symplectic_multi},\cite{J_HONG}). As one of the best-understood examples of singular projective varieties, the Schubert variety plays an important role in the study of generalized Grassmannian/flag varieties. Its relation with the cohomology theory on Grassmannian was first proposed by Hermann Schubert as early as the 19th century and later featured as the 15th problem among Hilbert's famous 23 problems.

Let $k$ be an algebraically closed field with $\hbox{char}(k)=0$, and $e_1,e_2,\cdots,e_n$ the standard basis of the linear space $k^n$. For any $d\leq n$, put
$$\begin{array}{rl}
I_{d,n}&=\{\underline{i}=(i_1,i_2,\cdots,i_d)|1\leq i_1<i_2<\cdots<i_d\leq n\}\\&=\{d-\mbox{subsets of } \{1,2,\cdots,n\}\}.\end{array}$$ Then $\{e_{\underline{i}}=e_{i_1}\wedge e_{i_2}\wedge\cdots\wedge e_{i_d}|\underline{i}\in I_{d,n}\}$ forms a basis of $\wedge^d k^n$. Denote its dual basis in $(\wedge^d k^n)^*$ by $\{p_{\underline{i}}|\underline{i}\in I_{d,n}\}$: $$p_{\underline{i}}(e_{\underline{j}})=1 \mbox{ if $\underline{i}=\underline{j}$; }0 \mbox{ if $\underline{i}\neq\underline{j}$}.$$ Then, $\{p_{\underline{i}}|\underline{i}\in I_{d,n}\}$ can be viewed as the homogeneous (projective) coordinates on $\mathbb{P}(\wedge^d k^n)$. Let $Gr(d,n)$ be the Grassmannian variety formed by the $d$-dimensional subspace of $k^n$ (if in the scheme-theoretical language, the closed points only). There is the famous Pl$\ddot{u}$cker embedding $$\begin{array}{cl}\hbox{Gr}(d,n)&\to \mathbb{P}(\wedge^d k^n)=\mathbb{P}^{\binom{n}{d}-1}\\\hbox{Span}\{v_1,v_2,\cdots,v_d\}&\mapsto [v_1\wedge v_2\wedge\cdots \wedge v_d].\end{array}$$
Additionally, $\{p_{\underline{i}}|\underline{i}\in I_{d,n}\}$ can be regarded as the homogeneous coordinates on $Gr(d,2n)$. Everything discussed in this article is under this Pl$\ddot{u}$cker embedding.

Let $$J=\left[\begin{matrix}
                 &  &  &  0&  & 1 \\
                 &  &  &  & \cdots &  \\
                 &  &  & 1 &  & 0 \\
                0 &  & -1 &  &  &  \\
                 & \cdots &  &  &  &  \\
                -1 &  & 0 &  &  &
              \end{matrix}\right]_{2n\times 2n}$$ Then, the symplectic Grassmannian is exactly $$\hbox{Gr}^C(d,2n)=\{V\in \hbox{Gr}(d,2n)|V\perp JV\},1\leq d\leq n,$$ where, for the column vectors $u,v\in k^{2n}$, $u\perp v$ means $u^Tv=0$. We employ the superscript $C$ because of its connection to the classic linear algebraic group of type C, i.e., the symplectic group $\mathrm{Sp}_{2n}$. In contrast, we sometimes use a superscript $A$ for objects corresponding to the special linear group $\mathrm{SL}_{2n}$.

Similarly, we can define the (partial) flag variety and symplectic flag variety:
$$\hbox{Fl}_{2n}(1,2,\cdots,n)=\{0\subset V_1\subset V_2\subset \cdots\subset V_n\subset k^{2n}|\dim V_t=t\},$$
$$\hbox{Fl}_{2n}^C(1,2,\cdots,n)=\{0\subset V_1\subset V_2\subset \cdots\subset V_n\subset k^{2n}|\dim V_t=t,V_n\perp JV_n\}.$$

In $\hbox{Gr}(d,2n)$ (resp. $\hbox{Fl}_{2n}(1,2,\cdots,n)$),  $\hbox{Gr}^C(d,2n)$ (resp. $\hbox{Fl}_{2n}^C(1,2,\cdots,n)$) is a closed subvariety. We can identify them with the homogeneous spaces of $\mathrm{SL}_{2n}$ and $\mathrm{Sp}_{2n}$, respectively. Let us fix the upper triangular Borel subgroups of $\mathrm{SL}_{2n}$ and $\mathrm{Sp}_{2n}$. Then, by taking the closures (under Zariski topology) of orbits of Borel action, we can obtain the  Schubert varieties of types A and C; see Section 3 for more details.

For the comparative study between Schubert varieties of types A and C, especially on the geometric aspects such as smoothness and singular locus, one widely used method is viewing $\mathrm{Sp}_{2n}$ as a fixed point set of a certain involution on $\mathrm{SL}_{2n}$ (cf. \cite{LAKSHMIBAI1987403GPVII},\cite{LAK_B},\cite{LAKSHMIBAI1997332},\cite[p.29]{billey_singular_2000}). And another technique, which is equally straightforward but has received less attention, is to use the defining ideals of varieties. The defining ideal is formed by objects in the homogeneous coordinate ring of Schubert variety of type A that are identically zero on Schubert variety of type C. In contrast to discussing the required conditions in the set-theoretical meaning, the discussion of obtaining variety of type C from  the intersection of variety of type A and giving the defining ideal is complicated. It is also a classically important problem to prove that some natural set-theoretical conditions generate the defining ideal: one example is on the determinantal variety formed by $m\times n$ matrices of rank not greater than $t$, where $n,m,t\in \mathbb{N}^+$. Clearly, the rank of a matrix is not greater than $t$ if and only all of its $t+1$-minors are zero. However, to explain why the defining ideal of such a determinantal variety is generated by $t+1$-minor functions (or equivalently these functions generate a radical ideal), we need a relatively long discussion\cite[p.180]{billey_singular_2000}. In our issue, the type C Grassmannian has been recognized as a linear slice of type A Grassmannian in many elementary research before, i.e. one has known that the defining ideal of type C Grassmannian is generated by linear polynomials\cite{fulton}. But a set of generators, which are naturally interpreted and easy to apply, and the analysis on it are still scarce.

So in our work, the problems are to describe the symplectic orthogonal condition ``$V\perp JV$" via some proper formulas and to determine whether these conditions generate the defining ideal of  Grassmannian/flag/Schubert varieties of type C in the corresponding varieties of type A. An existing relevant work on this topic is from De Concini\cite{DeConcini79}. In his study of the basis for coordinate rings of symplectic determinantal variety, he proposed an algorithm that indicated that the defining ideal could be generated by a series of special linear equations \cite[(1.8)]{DeConcini79} (but still not symplectic orthogonal conditions, at least a good interpretation is lacking), We refer to this result in further depth in Section 3 (see also \cite{Hilbert_sym} and Lakshmibai's work on admissible pair\cite{Lakshmibai2008} for more comments, which also reveals that De Concini's equations were considered to be a most advanced existing result for a long time in the past). Although De Concini has not developed a better discussion on those equations because of his focus on the basis for the coordinate ring, he did spark interest in various problems: the (local) orthogonal relation appears to be a quadratic equation rather than his linear one intuitively. While we use a different setting of equations than did De Concini, this problem will be clarified in the next section. With a natural interpretation, we obtain the expression of (set-theoretical) defining equations of  Grassmannian/flag/Schubert variety of type C in Grassmannian/flag/Schubert variety of type A as linear polynomials in the global homogeneous coordinates $p_{\underline{i}}$. Then, in the Grassmannian, we show that the restrictions of these equations generate the (scheme-theoretical) defining ideal of  Schubert variety of type C in the  varieties of type A. For flag varieties, a similar analysis is also conducted.

We should note that the equations in Section 2 also appeared in the cryptography research paper \cite{code} from Jesús Carrillo-Pacheco and Felipe Zaldivar, 2011. The main result of their work is to prove that the equations set-theoretically define the Lagrangian Grassmannian $Gr^C(n,2n)$ in $Gr(n,2n)$. In contrast, we obtain the equations in different methods from theirs and discuss on a more generalized scope rather than only Lagrangian Grassmannian. Further, the most important novelties of ours in this part are the natural interpretation to the equations and the scheme-theoretical conclusions in Section 3.

After determining the defining ideal,  the restriction of some of our equations on the Schubert varieties are zero. Then, counting the number of nonzero equations has its own importance: it provides a tool for further comparative studies between the Schubert varieties of types A and C. For an algebraic variety $X$ and its subvariety $Y$, we say $Y$ is a complete intersection of $X$ if the defining ideal of $Y$ in $X$ is generated by $codim_X Y$ equations. Moreover, $Y$ is called a local complete intersection of $X$ if there is an open cover $X=\cup U_i$ such that $Y\cap U_i$ is a complete intersection of $X\cap U_i$ for every $i$. Especially, at this time, if $X$ is smooth, then it is known that $Y$ is Gorenstein and Cohen–Macaulay, which implies that $Y$ has mild singularities. Being a local complete intersection of one smooth variety is an intrinsic property, so at this time, we can drop the emphasis on $X$. By counting the nonzero generators of the defining ideal, we can give the conditions for  Schubert varieties of type C to become a complete intersection of  variety of type A (or intrinsically).

Finally, we discuss a research gap related to the smoothness of Schubert varieties in $Gr^C(d,2n),1<d<n$. Existing studies on the symplectic Grassmannian concentrate on $Gr^C(n,2n)$ and trivial $Gr^C(1,2n)$ because their corresponding maximal parabolic subgroups in $\mathrm{Sp}_{2n}$ are the only cominuscule or minuscule ones\cite{LAKSHMIBAI1990179}. However, every maximal parabolic subgroup of $\mathrm{SL}_{2n}$ is cominuscule and minuscule, and the research on $Gr(d,2n)$ is fairly mature for any $1\leq d\leq 2n$. Thus, the defining ideal helps us to establish the connection between the Schubert varieties in $Gr(d,2n)$ and to those in $Gr^C(d,2n), 1<d<n$. On $Gr^C(n,2n)$, there is an interesting result that the symplectic Schubert variety must be smooth if the corresponding Schubert variety of type A is already smooth\cite{LAKSHMIBAI1990179}. A similar argument holds for symplectic Schubert varieties in flag variety\cite{LAKSHMIBAI1997332}. However, we note that this does not always hold for Schubert varieties in $Gr^C(d,2n),1<d<n$.

The rest of this paper is organized as follows:

In section 2, we collect preliminaries on Grassmannian variety, then provide a sketch of the equations we require. As a special result, the local relations, Proposition \ref{local relation}, are computed.

In section 3, we consider flag varieties and Schubert varieties. We introduce the work of De Concini, which enables us to focus on linear equations rather than higher degree ones. Then, employing an inductive method that might also be useful for further research, the main Theorems \ref{main1}, \ref{main2} are proved. The generators of the defining ideal of  Schubert varieties of type C in the varieties of  type A are identified. Moreover, utilizing the local relations provided in section 2, we obtain the local defining ideals.

In section 4,  the number of nonzero local equations on specific Schubert varieties is considered. Proposition \ref{complete intersection in Ai} reveals that in some special affine open subset, the number is exactly equal to the codimension, but in general, this finding does not hold. The relation among this number in different affine open subsets is discussed in Lemma \ref{number in affines lemma} and its Corollary \ref{number in affines}. As the main result of this section, we find the precise number of nonzero local equations in Proposition \ref{specific number}, and the conditions that are required for the  Schubert variety of type C to be a local complete intersection in the variety of  type A are derived in Theorem \ref{main3}.

In section 5, the symplectic conditions on tangent spaces are discussed. We compute the codimension of tangent spaces, and a result on the smoothness of Schubert variety in $Gr^C(d,2n),1<d<n,$ is noted. This result shows a difference between Schubert varieties in $Gr^C(d,2n),1<d<n$ and the ones in $\mathrm{Sp}_{2n}/Q$, where $Q$ is a Borel subgroup or minuscule/cominuscule maximal parabolic subgroup.
\section{Symplectic Grassmannian and flag varieties}
\subsection{Grassmannian and Pl$\ddot{u}$cker coordinates}
For all $d,n,t$ and $\underline{i}\in I_{d,n}$, we denote the $t$-th element in ascending order in $\underline{i}$ by $i_t$.

If $V\in Gr(d,n)$ has basis $v_1,v_2,$ $\cdots,v_d$ (as column vectors), then we can simultaneously determine all the projective coordinates $$p_{\underline{i}}(V)=\det(M_{i_1i_2\cdots i_d})$$, where $M$ is the matrix $[v_1,v_2,\cdots,v_d]$ and $M_{i_1i_2\cdots i_d}$ takes the $i_1,i_2,\cdots,i_d$ rows of $M$. As projective coordinates, their quotient $p_{\underline{i}}/p_{\underline{j}}$ remains invariant if we choose a different basis of $V$(with $p_{\underline{j}}(V)\neq 0$).

Note that if $p_{\underline{j}}(V)\neq 0$, then there must be a suitable basis of $V$ that makes $M_{j_1j_2\cdots j_d}=Id_{d\times d}$. We call such $M$ $\underline{j}$-standard. Additionally, if $M$ is $\underline{j}$-standard, we can directly compute $$\frac{p_{\underline{i}}}{p_{\underline{j}}}(V)=det(M_{i_1i_2\cdots i_d}).$$

Let $$J=\left[\begin{matrix}
                 &  &  &  0&  & 1 \\
                 &  &  &  & \cdots &  \\
                 &  &  & 1 &  & 0 \\
                0 &  & -1 &  &  &  \\
                 & \cdots &  &  &  &  \\
                -1 &  & 0 &  &  &
              \end{matrix}\right]_{2n\times 2n}.$$ The symplectic group $\mathrm{Sp}_{2n}=\{M\in \mathrm{SL}_{2n}|M^TJM=J\}$. We fix the group of upper triangular matrices $B^A$ (resp. $B^C$) in $\mathrm{SL}_{2n}$ (resp. in $\mathrm{Sp}_{2n}$) as the standard Borel subgroup and the group of diagonal matrices as the maximal torus in $B^A$ (resp. in $B^C$). Take $\epsilon_t(diag(a_1,a_2,\cdots,a_{2n}))=a_t$ ($1\leq t\leq 2n$ for $\mathrm{SL}_{2n}$ and $1\leq t\leq n$ for $\mathrm{Sp}_{2n}$) to be the canonical basis for multiplicative characters of the maximal torus. With respect to this maximal torus and standard Borel, we have the simple roots $S(\mathrm{SL}_{2n})=\{\epsilon_t-\epsilon_{t+1}|1\leq t\leq 2n-1\}$ of $\mathrm{SL}_{2n}$, and $S(\mathrm{Sp}_{2n})=\{\epsilon_t-\epsilon_{t+1}|1\leq t\leq n-1\}\cup \{2\epsilon_n\}$ of $\mathrm{Sp}_{2n}$.

Associated with every subset of simple roots $S'$, a unique parabolic subgroup containing the standard Borel can be determined\cite[p.147]{springer_linear_1998}. For $1\leq d\leq n$, let $P^A_d$ be the parabolic subgroup in $\mathrm{SL}_{2n}$ associated with $S(\mathrm{SL}_{2n})\backslash\{\epsilon_d-\epsilon_{d+1}\}$  (then $P^A_d$ is the group of ``$(d,n-d)$-blocked" upper triangular matrices), and let $P^C_d$ be the parabolic subgroup in $\mathrm{Sp}_{2n}$ associated with the set $$\left\{\begin{array}{ll}
                                                                   S(\mathrm{Sp}_{2n})\backslash \{\epsilon_d-\epsilon_{d+1}\} & ,\mbox{ if }1\leq d\leq n-1 \\
                                                                   S(\mathrm{Sp}_{2n})\backslash \{2\epsilon_n\} & ,\mbox{ if }d=n
                                                                 \end{array}\right.$$

Note that $P^C_d=P^A_d\cap \mathrm{Sp}_{2n}$; then, we have the induced embedding $\mathrm{Sp}_{2n}/P^C_d\to \mathrm{SL}_{2n}/P^A_d$. We have the quotients $\mathrm{SL}_{2n}/P^A_d \cong Gr(d,2n)$ as homogeneous $\mathrm{SL}_{2n}$-space\cite[p.168]{lakshmibai_flag_2018}. Under this identification,  $$\mathrm{Sp}_{2n}/P^C_d \cong Gr^C(d,2n)=\{U\in Gr(d,2n)|V\perp JV\}.$$

We can quickly determine the following dimension formula from the root data.
\begin{prop}
  $\dim Gr^C(d,2n)=\dim Gr(2,2n)-\binom{d}{2}, d\leq n.$
\end{prop}

\subsection{The symplectic conditions and their local relations }

For any sequence of integers $a_1,a_2,\cdots,a_l$, let $\tau(a_1,a_2,\cdots,a_l)=\#\{1\leq s<t\leq l|a_s>a_t\}$ be the inversion number. Our first goal is to show that

\begin{theo*}
  $Gr^C(d,2n)$ is the (set-theoretical) intersection of $Gr(d,2n)$ with hyperplanes
  $$E_{\underline{i’}}=\sum_{t=1}^{n}(-1)^{\tau(i_1',i_2',\cdots,i_{d-2}',t,2n+1-t)}p_{\underline{i'}\cup \{t,2n+1-t\}},\underline{i'}\in I_{d-2,2n}.$$
\end{theo*}

Note that $\underline{i'}$ can be viewed as a subset of $\{1,2,\cdots,2n\}$, so the notation $\underline{i'}\cup \{t,2n+1-t\}$ makes sense. If the subscript of the Pl$\ddot{u}$cker coordinate has overlap, i.e., $\# (\underline{i'} \cup \{t,2n+1,t\})<d$, we set this term (also the inversion number) to zero.

In $\mathbb{P}(\wedge^d k^{2n})$ with homogeneous coordinates $\{p_{\underline{i}}\mid \underline{i}\in I_{d,2n}\}$, the canonical affine open cover of $Gr(d,2n)$ is formed by $A_{\underline{i}}=\{V\in Gr(d,2n)\mid p_{\underline{i}}(V)\neq 0\}\cong \mathbb{A}^{d(2n-d)},\underline{i}\in I_{d,n}$. Now, we consider the affine open subset $A_{\underline{i}}$ and the local relations of hyperplanes mentioned above.

For a matrix $M_{n\times d}$, denote its columns by $c_1,c_2,\cdots,c_d$, and set $$C(M,s,t)=c_s^TJc_t.$$ Clearly, a vector space $V\in Gr(d,2n)$ belongs to $Gr^C(d,2n)$ if for all matrix presentation $M$, $\forall s\neq t,C(M,s,t)=0$. Moreover, note that $C(M,s,s)=0, C(M,s,t)=-C(M,t,s)$ for all $s,t$.

But locally, for example, in $A_{id},id=(1,2,\cdots,d)\in I_{d,2n}$, the equations $C(M,s,t)$ do not take linear forms in local coordinates.
\begin{examp}$d=3,n=4,$ (without confusion we omit the comma ',' in $\underline{i}$)
  $$M=\left[\begin{matrix}
             1 & 0 & 0 \\
             0 & 1 & 0 \\
             0 & 0 & 1 \\
             x_{41} & x_{42} & x_{43} \\
             x_{51} & x_{52} & x_{53} \\
             x_{61}& x_{62} & x_{63} \\
             x_{71} & x_{72} & x_{73} \\
             x_{81} & x_{82} & x_{83}
           \end{matrix}\right]$$is $id$-standard. $C(M,1,2)=x_{82}+x_{41}x_{52}-x_{71}-x_{42}x_{51}$ is of degree $2$ in $x_{ij}$. Meanwhile, we can write this equation as:
           $$C(M,1,2)=-\frac{p_{138}}{p_{123}}+\frac{p_{345}}{p_{123}}-\frac{p_{237}}{p_{123}}=\frac{E_3}{p_{123}}.$$

  Or equivalently, $$x_{41}x_{52}-x_{42}x_{51}=-\frac{p_{234}p_{135}}{p_{123}^2}-\frac{-p_{134}p_{235}}{p_{123}^2}=\frac{p_{345}p_{123}}{p_{123}^2}.$$This nonhomogeneous equation provides a common factor $p_{123}$ in the form of global coordinates.
\end{examp}

\begin{prop}\label{C}
  If $V\in A_{\underline{i}}$ and $M$ is an $\underline{i}$-standard matrix presentation of $V$, then $\forall s<t$, we have
  $$C(M,s,t)=(-1)^{s+t}\frac{E_{\underline{i'}}}{p_{\underline{i}}}(V),$$
  where $\underline{i'}=\underline{i}\backslash \{i_s,i_t\}\in I_{d-2,2n}$, and $$E_{\underline{i’}}=\sum_{t=1}^{n}(-1)^{\tau(i_1',i_2',\cdots,i_{d-2}',t,2n+1-t)}p_{\underline{i'}\cup \{t,2n+1-t\}},\underline{i'}\in I_{d-2,2n}.$$
\end{prop}

\begin{proof}
  Note that $C(M,s,t)=-C(M,t,s)$, so in this proposition, we  discuss only the case $s<t$. Then, $\tau(\underline{i'},i_s,i_t)=2d-(s+t)$. Denote the entries of $M$ by $x_{ij}$, and set $C_k(M,s,t)=x_{ks}x_{2n+1-k,t}-x_{2n+1-k,s}x_{kt}$ for $1\leq k\leq n$. Then, $C(M,s,t)=\sum_{k=1}^{n}C_k(M,s,t)$.

  Now, consider $C_k(M,s,t)$.

  \begin{itemize}
    \item[case 1.]$k\not\in \underline{i}$ and $2n+1-k\not\in \underline{i}$, i.e. $\{k,2n+1-k\}\cap \underline{i}=\emptyset$. Then,
    $$C_k(M,s,t)=x_{ks}x_{2n+1-k,t}-x_{2n+1-k,s}x_{kt}$$$$=(-1)^{\tau(\underline{i'},i_s,i_t)+\tau(\underline{i'},k,2n+1-k)}\frac{p_{\{k,2n+1-k\}\cup \underline{i'}}}{p_{\underline{i}}}(V).$$
    \item[case 2.]$\{k,2n+1-k\}\cap\underline{i}\neq\emptyset$, but $\{k,2n+1-k\}\cap \{i_s,i_t\}=\emptyset$. Then,
    $$C_k(M,s,t)=0=(-1)^{\tau(\underline{i'},i_s,i_t)+\tau(\underline{i'},k,2n+1-k)}\frac{p_{\{k,2n+1-k\}\cup \underline{i'}}}{p_{\underline{i}}}(V),$$becase $\{k,2n+1-k\}\cap\underline{i'}\neq \emptyset$.
    \item[case 3.]$\{k,2n+1-k\}\cap\{i_s,i_t\}\neq \emptyset$. If $k=i_s$, then no matter whether $2n+1-i_s=i_t$ or not,
    $$C_k(M,s,t)=x_{2n+1-i_s,t}=(-1)^{\tau(\underline{i'},i_s,i_t)+\tau(\underline{i'},i_s,2n+1-i_s)}\frac{p_{\{i_s,2n+1-i_s\}\cup \underline{i'}}}{p_{\underline{i}}}(V).$$
    If $k=i_t$, then
    $$C_k(M,s,t)=-x_{2n+1-i_t,s}=(-1)^{\tau(\underline{i'},i_s,i_t)+\tau(\underline{i'},2n+1-i_t,i_t)+1}\frac{p_{\{i_t,2n+1-i_t\}\cup \underline{i'}}}{p_{\underline{i}}}(V)$$$$=(-1)^{\tau(\underline{i'},i_s,i_t)+\tau(\underline{i'},i_t,2n+1-i_t)}\frac{p_{\{i_t,2n+1-i_t\}\cup \underline{i'}}}{p_{\underline{i}}}(V).$$

    If $2n+1-k=i_s$, then
    $$C_k(M,s,t)=-x_{2n+1-i_s,t}=(-1)^{\tau(\underline{i'},i_s,i_t)+\tau(\underline{i'},2n+1-i_s,i_s)+1}\frac{p_{\{i_s,2n+1-i_s\}\cup \underline{i'}}}{p_{\underline{i}}}(V)$$ $$=(-1)^{\tau(\underline{i'},i_s,i_t)+\tau(\underline{i'},2n+1-i_s,i_s)}\frac{p_{\{i_s,2n+1-i_s\}\cup \underline{i'}}}{p_{\underline{i}}}(V)$$
    If $2n+1-k=i_t$, then
    $$C_k(M,s,t)=x_{2n+1-i_t,s}=(-1)^{\tau(\underline{i'},i_s,i_t)+\tau(\underline{i'},2n+1-i_t,i_t)}\frac{p_{\{i_t,2n+1-i_t\}\cup \underline{i'}}}{p_{\underline{i}}}(V).$$
  \end{itemize}
Thus, for $\forall 1\leq k\leq n$, $$C_k(M,s,t)=(-1)^{\tau(\underline{i'},i_s,i_t)+\tau(\underline{i'},k,2n+1-k)}\frac{p_{\{k,2n+1-k\}\cup \underline{i'}}}{p_{\underline{i}}}(V).$$
$$C(M,s,t)=\sum_{k=1}^{n}C_k(M,s,t)=(-1)^{\tau(\underline{i'},i_s,i_t)}\frac{E_{\underline{i'}}}{p_{\underline{i}}}(V)=(-1)^{s+t}\frac{E_{\underline{i'}}}{p_{\underline{i}}}(V).$$
\end{proof}
Since $A_{\underline{i}},\underline{i}\in I_{d,2n}$ forms an open cover of $Gr(d,2n)$, we obtain
\begin{theo}\label{hyperplane}
  $Gr^C(d,2n)$ is the (set-theoretical) intersection of $Gr(d,2n)$ with hyperplanes:
  $$E_{\underline{i’}}=\sum_{t=1}^{n}(-1)^{\tau(i_1',i_2',\cdots,i_{d-2}',t,2n+1-t)}p_{\underline{i'}\cup \{t,2n+1-t\}},\underline{i'}\in I_{d-2,2n}.$$
\end{theo}
\begin{prop}\label{local relation}
  $\forall \underline{j'}\in I_{d-2,2n}$ and $\underline{i}\in I_{d,2n}$,
  $$E_{\underline{j'}}=\sum_{l=1}^{d}\sum_{k=1,k<l}^{d}\frac{E_{\underline{i}\backslash\{i_k,i_l\}}}{p_{\underline{i}}}\cdot(-1)^{k+l+\tau(\underline{j'},i_k,i_l)} p_{\underline{j'}\cup\{i_k,i_l\}}$$
  on $A_{\underline{i}}$, the affine open subset of $Gr(d,2n)$.
\end{prop}

\begin{proof}
Extend $\underline{j'}$ to some $\underline{j}=\{q,r\}\cup \underline{j'}\in I_{d,2n}$, i.e., add two proper indices $q<r$ into $\underline{j'}$.

Now, consider $V\in A_{\underline{i}}\cap A_{\underline{j}}$. Let $M=(x_{ij})$ be a $\underline{j}$-standard matrix presentation of $V$ and $M_{\underline{i}}\neq 0$. Denote by $M_{\underline{i}}$ the submatrix formed by the $\underline{i}$-th row of $M$ and similarly for any index in $I_{d,2n}$. Then, $N=MM_{\underline{i}}^{-1}$ is $\underline{i}$-standard. Let the columns of $M$ be $c_1,c_2,\cdots,c_d$ and the columns of $N$ be $\tilde{c_1},\tilde{c_2},\cdots,\tilde{c_d}$.
$$M=[c_1,c_2,\cdots,c_d]=[\tilde{c_1},\tilde{c_2},\cdots,\tilde{c_d}]M_{\underline{i}}$$
Let $s=\#\{a\in \underline{j},a\leq q\}$ and $t=\#\{a\in \underline{j},a\leq r\}$, $$\begin{array}{c}C(M,s,t)=c_s^TJc_t\\=(\sum_{k=1}^{d} x_{i_k,s}\tilde{c_k}^T)J(\sum_{l=1}^{d} x_{i_l,t}\tilde{c_l})=\sum_{l=1}^{d}\sum_{k=1}^{d} x_{i_k,s}x_{i_l,t}C(N,k,l).\end{array}$$ Note that $C(N,k,l)=0$ if $k=l$ and $C(N,k,l)=-C(n,l,k)$ for any $k,l$. Then,
$$C(M,s,t)=\sum_{l=1}^{d}\sum_{k=1,k<l}^{d} (x_{i_k,s}x_{i_l,t}-x_{i_l,s}x_{i_k,t})C(N,k,l)$$$$= (-1)^{\tau(\underline{j'},j_s,j_t)+\tau(\underline{j'},i_k,i_l)}\frac{p_{\underline{j'}}\cup\{i_k,i_l\}}{p_{\underline{j}}}(V)\cdot C(N,k,l)$$

By Proposition \ref{C}, since $M$ is $\underline{j}$-standard, $C(M,s,t)=(-1)^{\tau(\underline{j'},j_s,j_t)}\frac{E_{\underline{j'}}}{p_{\underline{j}}}(V)$. Additionally, for $\underline{i}$-standard $N$, $C(N,k,l)=(-1)^{\tau(\underline{i'},i_k,i_l)}\frac{E_{\underline{i'}}}{p_{\underline{i}}}(V)$. Combining these results, we obtain
$$E_{\underline{j'}}=\sum_{l=1}^{d}\sum_{k=1,k<l}^{d}\frac{E_{\underline{i}\backslash\{i_k,i_l\}}}{p_{\underline{i}}}\cdot (-1)^{k+l+\tau(\underline{j'},i_k,i_l)}p_{\underline{j'}\cup\{i_k,i_l\}}$$

\end{proof}

\begin{coro}\label{local hyperplanes}
  On affine open $A_{\underline{i}}, \underline{i}\in I_{d,2n}$, symplectic Grassmannian $Gr^C(d,2n)$ is the set-theoretical intersection of $Gr(d,2n)$ with hyperplanes
  $\frac{E_{\underline{i'}}}{p_{\underline{i}}},$ where $\underline{i'}$ takes all $\underline{i'}\in I_{d-2,2n}$ and $\underline{i'}\subset \underline{i}$. These hyperplanes form a minimal generating set of the ideal generated by them in the coordinate ring $k[A_{\underline{i}}]$.
\end{coro}
\begin{proof}
   It is clear from Proposition \ref{local relation} and that $$\#\{\underline{i'}\in I_{d-2,2n},\underline{i'}\subset\underline{i}\}=\binom{d}{2}=codim_{Gr(d,2n)}Gr^C(d,2n).$$
\end{proof}
\subsection{Symplectic flag varieties}
Generally, a flag is some nesting sequence of subspaces of linear space $k^n$ for given $s,n\in \mathbb{N}$:
$$0\subset V_1\subset \cdots\subset V_s\subset k^n,$$
where $V_t$ is a proper subspace of $V_{t+1}$ for any $1\leq t\leq s-1$. For any $(i_1,i_2,\cdots,i_s)\in I_{s,n}$, the flag variety $Fl_{n}(i_1,i_2,\cdots,i_s)$ is defined as the collection of flags $$\{\underline{V}=(V_1,V_2,\cdots,V_s)|0\subset V_1\subset \cdots\subset V_s\subset k^n, \dim V_t=i_t,\forall 1\leq t\leq s\}.$$ If $s=n$, then $\dim V_t=t$ for all $1\leq t\leq n$, and such $Fl_{n}(1,2,\cdots,n)$ is called a complete flag variety. Otherwise, $Fl_{n}(i_1,\cdots,i_s)$ is a partial flag variety.

As with Grassmannian varieties, we can identify $\mathrm{SL}_{2n}/B^A$ with the flag variety $Fl_{2n}(1,2,\cdots,2n)$ and $\mathrm{SL}_{2n}/P^A_{12\cdots n}$ with the partial flag variety $Fl_{2n}(1,2,$ $\cdots,n)$, where $P^A_{12\cdots n}=P^A_1\cap P^A_2\cap\cdots\cap P^A_n$. Since for the Borel subgroup $B^C$ of $\mathrm{Sp}_{2n}$, we have $B^C=B^A\cap \mathrm{Sp}_{2n}= P^A_{12\cdots n}\cap \mathrm{Sp}_{2n}$, we can embed the quotient $\mathrm{Sp}_{2n}/B^C$ into $\mathrm{SL}_{2n}/P^A_{12\cdots n}$, factoring through the canonical surjection $\mathrm{SL}_{2n}/B^A\to \mathrm{SL}_{2n}/P^A_{12\cdots n}$. This correspondence identifies $\mathrm{Sp}_{2n}/B^C$ with $$\{(V_1,V_2,\cdots,V_{2n})\in Fl_{2n}(1,2,\cdots,2n)|V_t\perp JV_{2n-t},\forall 1\leq t\leq n\}\subset \mathrm{SL}_{2n}/B^A $$and also with $$\{(V_1,V_2,\cdots,V_n)\in Fl_{2n}(1,2,\cdots,n)|V_n\perp JV_n\}\subset \mathrm{SL}_{2n}/P^A_{12\cdots n}.$$ In this section, we are concerned mainly with the second identification. Denote $$\mathrm{Sp}_{2n}/B^C=Fl_{2n}^C(1,2,\cdots,n)=\{(V_1,V_2,\cdots,V_n)\in Fl_{2n}(1,2,\cdots,n)|V_n\perp JV_n\}.$$

Next, let us introduce the homogeneous coordinates and the affine open cover on $Fl_{2n}(1,2,\cdots,n)$. For a projective variety $X$, let $\Gamma_h(X)$ be its homogeneous coordinate ring and $\Gamma^t_h(X)$ be the $t$-th grading. For $f\in \Gamma_h(X)$, we say $f$ is a homogeneous section on $X$. 
  Given projective varieties $X,Y$ and surjective homomorphism of graded rings $\Phi: \Gamma_h(Y)\to \Gamma_h(X)$, there must be a unique morphism of projective varieties $\varphi:X\to Y$ such that the rational functions $\frac{\Phi(f)}{\Phi(g)}=\frac{f}{g}\circ\varphi$ for any $f,g\in\Gamma_h(Y)$ with $\deg f=\deg g$. At this time, we denote $\Phi=\varphi^*$ and for every $f\in \Gamma_h(Y),x\in X$, we note that $\varphi^*(f)(x)=0$ if and only if $f(\varphi(x))=0$. We say such $\varphi: X\to Y$ is induced by $\varphi^*: \Gamma_h(Y)\to \Gamma_h(X)$, but  to be intuitive, usually we declare the morphism of varieties $\varphi$ first. It is easy to check that, in the remainder of this paper, every  morphism of projective varieties is in this form. Moreover, if $\varphi$ is just the inclusion induced by $\varphi^*:\Gamma_h(Y)\to \Gamma_h(Y)/I=\Gamma_h(X)$ for some ideal $I\subset \Gamma_h(Y)$, we also denote $\varphi^*(f)$ by $f|_X$ for $f\in \Gamma_h(Y)$. Sometimes we are concerned with only the certain $1$-st grading, which determines the information on every grading, and we simply denote $\varphi^*:\Gamma_h^1(Y)\to\Gamma_h^{1}(X)$ as well.

$Fl_{2n}(1,2,\cdots,n)$ is a closed subvariety of $Gr(1,2n)\times Gr(2,2n)\times \cdots \times Gr(n,2n)$\cite[p.126]{lakshmibai_flag_2018}, which enables us to describe any homogeneous section in $\Gamma_h(Fl_{2n}(1,2,\cdots,n))$ as a polynomial in $\{p_{\underline{i}}|\underline{i}\in \cup_{d=1}^n I_{d,2n}\}$, which should be homogeneous in $\{p_{\underline{i}}| \underline{i}\in I_{d,2n}\}$ for every fixed $d$. On every Grassmannian $Gr(d,2n)$, we have defined a series of $E_{\underline{i'}},\underline{i'}\in I_{d-2,2n}$; now, they are also lying in $\Gamma_h(Fl_{2n}(1,2,\cdots,n))$. To avoid confusion about the notation, one should note the cardinal number $d$ in the subscript $\underline{i}$.

 For $\underline{V}=(V_1,V_2,\cdots,V_n)\in Fl_{2n}(1,2,\cdots,n)$, we can find a sequence of vectors $v_1,v_2,\cdots,v_n$ such that $V_t=Span\{v_1,\cdots,v_t\}$ for every $1\leq t\leq n$. Then $M_d=[v_1,\cdots,v_d]$ for $1\leq d\leq n$ is a matrix presentation of $V_d$ for $1\leq d\leq n$, and we say $M_n$ is a matrix presentation of $\underline{V}$. There must be a sequence $(w_1,w_2,\cdots,w_n)$, where $w_t$ are distinct in $\{1,2,\cdots,2n\}$, such that $p_{\{w_1,\cdots,w_t\}}(\underline{V})=p_{\{w_1,\cdots,w_t\}}(V_t)\neq 0$ . Consider $$\Delta_{n,2n}=\{(w_1,w_2,\cdots,w_n)|w_1,w_2,\cdots,w_n\ \mbox{are\ distinct in } \{1,2,\cdots,2n\}\}.$$ and for $\forall \underline{w}=(w_1,w_2,\cdots,w_n)\in \Delta_{n,2n}$ (note that $\underline{w}$ is a sequence, can not be rearranged like $\underline{i}\in I_{d,2n}$), let $$\underline{w}^{(d)}=\{w_1,w_2,\cdots,w_d\}\in I_{d,2n}, 1\leq d\leq n$$ $$O_{\underline{w}}=\{\underline{V}\in Fl_{2n}(1,2,\cdots,n)|p_{\underline{w}^{(d)}}(\underline{V})\neq 0,\forall 1\leq d\leq n\}.$$ Then, all the $O_{\underline{w}},\underline{w}\in \Delta_{n,2n}$ form an affine (isomorphic to $\mathbb{A}^{\frac{3n^2-n}{2}}$) open cover of $Fl_{2n}(1,2,\cdots,n)$.

\begin{prop}\label{flag E}
  Let $\underline{w}\in \Delta_{n,2n}$. If $\underline{i}=\underline{w}^{(d_1)}$ and $\underline{j}=\underline{w}^{(d_2)}$ with $d_1\leq d_2$, then $\forall s,t\leq d_1$
  $$\frac{E_{\underline{i}\backslash\{i_s,i_t\}}}{p_{\underline{i}}}=\pm\frac{E_{\underline{j}\backslash\{j_{s'},j_{t'}\}}}{p_{\underline{j}}}$$on $O_{\underline{w}}$ for some $s',t'\leq d_2$.
\end{prop}
\begin{proof}
  We only need to prove for $d_2=d_1+1$. Consider $\forall \underline{V}\in O_{\underline{w}}.$ There must be an $\underline{i}$-standard matrix presentation of $V_{d_1}$. Then, we can extend the columns of this matrix to a basis of $V_{d_2}$ by means of a vector $v$ with the $t$-th component equal to zero for all $t\in \underline{i}$. Additionally, the $\underline{j}\backslash\underline{i}$-th component of $v$ must be nonzero because $p_{\underline{j}}(V_{d_2})\neq 0$. Up to a column permutation, we obtain a $\underline{j}$-standard matrix presentation of $V_{d_2}$. By Proposition \ref{C}, we can complete the proof.
\end{proof}
\section{Symplectic Schubert varieties in Grassmannian and flag variety}
Let $G$ be a connected semisimple algebraic group\cite[Ch.6-8]{springer_linear_1998}. We fix a maximal torus $T$ and a standard Borel $B\supset T$ in $G$. Let $W(G)=N_G(T)/T$, the Weyl group of $G$ with respect to $T$, which is generated by the reflection of simple roots. We can define the length function $l(\cdot)$ on $W(G)$ . For any parabolic subgroup $P$ of $G$ containing $B$ and a given $\theta \in W(G)$ with representative $\dot{\theta}\in N_G(T)$, we denote the well-defined coset $\dot{\theta}P$ in $G/P$ by $e_\theta$. $B$ naturally acts on $G/P$ as a subgroup of $G$; then, there are orbits $Be_\theta$ for $\theta\in W(G)$. Moreover, for $\theta_1,\theta_2\in W(G)$, $Be_{\theta_1}=Be_{\theta_2}$ if and only if $\theta_1\in \theta_2W(P)$. Thus, it is also well-defined to say $Be_{\theta}$ for a coset $\theta W(P)\in W(G)/W(P)$. Furthermore, in every coset $\theta W(P)\in W(G)/W(P)$, there exists a unique element in minimal length (in $W(G)$). Collecting all such shortest representatives of cosets, we define
$$W(G)_P^{min}=\{\theta\in W(G)|l(\theta)\leq l(\theta'), \forall\theta'\in \theta W(P)\}.$$

 A Schubert variety $X(\theta)=\overline{Be_\theta}, \theta\in W(G)_P^{min}$ is defined to be the closure of $Be_\theta$ in $G/P$. This closure is also a union of $B$-orbits. We have the Bruhat order on $W(G)_P^{min}$ (also $W(G)/W(P)$): for $\theta',\theta\in W(G)_P^{min},$ if $Be_{\theta'}\subset X(\theta)=\overline{Be_\theta}$ or equivalently $e_{\theta'}\in X(\theta)$, we say $\theta'\leq \theta$. Thus, $X(\theta)=\cup_{\theta'\leq \theta} Be_{\theta'}$. Note that $G/P$ itself is also a special Schubert variety.

In the cases we are concerned about, let $G=\mathrm{SL}_{2n}$, $T$ be the diagonal maximal torus, $B=B^A$ and $P=P^A_d$. Then, $W(G)\cong S_{2n}$ is the symmetric group. Instead of the cyclic expression, we write every permutation in $S_{2n}$ in a one-line expression: $\sigma\in S_{2n}$, which we denote by $\sigma=(\sigma(1),\sigma(2),\cdots,\sigma(2n))$. Then, $W(P^A_d)\cong \{(\sigma(1),\sigma(2),\cdots,\sigma(2n))\in S_{2n}|\{\sigma(1),\sigma(2),\cdots,\sigma(d)\}=\{1,2,\cdots,d\}\}\cong S_{d}\times S_{2n-d}$. For $w\in S_{2n}$, the length is $l^A(w)=\tau(w)$, the inversion number of its one-line expression (and for $w\in W(\mathrm{Sp}_{2n}), l^C(w)=(\tau(w)+m)/2$, where $m=\#\{1\leq t\leq d|w_t>n\}$). That implies that we can identify $W(\mathrm{SL}_{2n})_{P^A_d}^{min}$ with $I_{d,2n}$:
$$(\sigma(1),\sigma(2),\cdots,\sigma(d),\cdots)\in W(\mathrm{SL}_{2n})_{P^A_d}^{min} \mapsto \underline{i}=(\sigma(1),\sigma(2),\cdots,\sigma(d))\in I_{d,2n}.$$ Under this identification, the notation $e_{\underline{i}}=\underline{i}P^A_d$ is consistent with our definition in Section 1.

Similarly, we have the following settings (the maximal torus is always the diagonal torus)(see also \cite[Ch.5]{billey_singular_2000}):
\begin{table}[htbp]
		\centering
		\begin{tabular}{|l|l|l|l|l|}\hline
$G $     & $B$  & $W(G)$  & $P$    & $W(G)_P^{min}$\\\hline
$\mathrm{SL}_{2n}$& $B^A$& $S_{2n}$& $P^A_d$& $I_{d,2n}$\\\cline{4-5}
         &       &        & $P^A_{12\cdots n}$ &$\Delta_{n,2n}$\\\hline
$\mathrm{Sp}_{2n}$& $B^C$ & $W(\mathrm{Sp}_{2n})$ & $P^C_d$ &$I^{Sp}_{d,2n}$\\\cline{4-5}
         &       &            &  $B^C$    & $W(\mathrm{Sp}_{2n})=\Delta^{Sp}_{n,2n}$\\\hline
		\end{tabular}
		\caption{The shortest representatives for different $G$, $B$ and $P$}
\end{table}

Here, $$W(\mathrm{Sp}_{2n})=\{(\sigma(1),\sigma(2),\cdots,\sigma(2n))\in S_{2n}|\sigma(t)+\sigma(2n+1-t)=2n+1,\forall 1\leq t\leq n\},$$ it can be identified with $$\Delta^{Sp}_{n,2n}=\{(w_1,w_2,\cdots,w_n)\in \Delta_{n,2n}|w_s+w_t\neq 2n+1,\forall s\neq t\}\subset \Delta_{n,2n}.$$

The Bruhat order on $I_{d,2n}$ (resp. $I^{Sp}_{d,2n}$) is given by(cf.\cite[Ch.5]{billey_singular_2000}): for $$\underline{i}=(i_1,i_2,\cdots,i_d),\underline{j}=(j_1,j_2,\cdots, j_d)\in I_{d,2n} \mbox{(resp. $I^{Sp}_{d,2n}$)}$$ with $i_1<i_2<\cdots<i_d,j_1<j_2<\cdots<j_d$, we have$\underline{i}\leq \underline{j}$ under the Bruhat order if and only if $i_t\leq j_t$ for every $1\leq t\leq d$. Additionally, the Bruhat order on $\Delta_{n,2n}$ (resp. $\Delta^{Sp}_{n,2n}$) is: for $\underline{w},\underline{w'}\in \Delta_{n,2n}$ (resp. $\Delta^{Sp}_{n,2n}$), $\underline{w}\leq\underline{w'}$ if and only if $\underline{w}^{(d)}\leq \underline{w'}^{(d)}$ in every $I_{d,2n}, 1\leq d\leq n$.

Moreover, in the language of subspaces, we can characterize the Schubert varieties in $\mathrm{SL}_{2n}/P^A_d=Gr(d,2n)$ as
$$\forall\underline{i}\in I_{d,2n},X^A(\underline{i})=\{V\in Gr(d,2n)|dim(V\cap Span\{e_1,e_2,\cdots,e_{i_t}\})\geq t,\forall 1\leq t\leq d\}.$$

The Schubert varieties in flag variety $\mathrm{SL}_{2n}/P^A_{12\cdots n}=Fl_{2n}(1,2,\cdots,n)$ are
$$\forall\underline{w}\in \Delta_{n,2n}, X^A(\underline{w})=\{\underline{V}\in Fl_{2n}(1,2,\cdots,n)|V_d\in X^A(\underline{w}^{(d)}),\forall 1\leq d\leq n\}.$$
\begin{lem}\cite[p.16]{billey_singular_2000},\cite[p.173]{lakshmibai_flag_2018}\label{ideal of type Aschubert}
  Let $\underline{i}\in I_{d,2n}$. The defining ideal of $X^A(\underline{i})$ in the homogeneous coordinate ring of $Gr(d,2n)$  is generated by $$\{p_{\underline{j}} |\underline{j}\in I_{d,2n},\underline{j}\not\leq \underline{i}\}.$$
  Let $\underline{w}\in \Delta_{n,2n}$. The defining ideal of $X^A(\underline{w})$ in the homogeneous coordinate ring of $Fl_{2n}(1,2,\cdots,n)$  is generated by $$\cup_{d=1}^n\{p_{\underline{j}} |\underline{j}\in I_{d,2n},\underline{j}\not\leq \underline{w}^{(d)}\}.$$
\end{lem}
The symplectic Schubert varieties in the Grassmannian and flag variety are, respectively,
$$\forall \underline{i}\in I^{Sp}_{d,2n}\subset I_{d,2n}, X^C(\underline{i})=X^A(\underline{i})\cap Gr^C(d,2n),$$ and
$$\forall \underline{w}\in \Delta^{Sp}_{n,2n}\subset \Delta_{n,2n}, X^C(\underline{w})=X^A(\underline{w})\cap \mathrm{Sp}_{2n}/B^C.$$
We should be careful that even though $I^{Sp}_{d,2n}\subset I_{d,2n}$ and there are identifications $W(\mathrm{Sp}_{2n})^{min}_{P^C_d}=I^{Sp}_{d,2n}, W(\mathrm{SL}_{2n})^{min}_{P^A_d}=I_{d,2n}$, in general, an index $\underline{i}\in I^{Sp}_{d,2n}$ corresponds to two different representatives  in $W(\mathrm{Sp}_{2n})^{min}_{P^C_d}$ and $W(\mathrm{SL}_{2n})^{min}_{P^A_d}$, as we will discuss in Section 4. However, viewing $I^{Sp}_{d,2n}$ the index set as a subset of $I_{d,2n}$ provides many conveniences, it only means that for any $\underline{i}\in I^{Sp}_{d,2n}$, it can be used as an index for both  Schubert varieties of types A and C in the Grassmannian (similarly in flag varieties).

Naturally, for every $d\leq n$, there is a projection $$pr_d: Fl_{2n}(1,2,\cdots,n)\to Gr(d,2n),\underline{V}=(V_1,V_2,\cdots,V_n)\mapsto V_d.$$ In a flag variety, a flag $\underline{V}\in Fl_{2n}(1,2,\cdots,n)$ belongs to $\mathrm{Sp}_{2n}/B^C$ if and only if $pr_n(\underline{V})\in Gr^C(n,2n)$. By Corollary \ref{local hyperplanes}, we immediately obtain the following.

\begin{theo}
  Let $\underline{i}\in I^{Sp}_{d,2n}$ (resp. $\underline{w}\in \Delta^{Sp}_{n,2n}$). The symplectic Schubert variety $X^C(\underline{i})$ (resp. $X^C(\underline{w})$) is the intersection of $X^A(\underline{i})$ (resp. $X^A(\underline{w})$) with hyperplanes $E_{\underline{i'}},\underline{i'}\in I_{d-2,2n}$ (resp. $\underline{i'}\in \cup_{d'=2}^{n}I_{d'-2,2n}$)
\end{theo}

However, it is not sufficient to say that $\{E_{\underline{i'}}|\underline{i'}\in I_{d-2,2n}\}$ generates the defining ideal $I_{X^A(\underline{i})}(X^C(\underline{i}))$ of $X^C(\underline{i})$ in $X^A(\underline{i}),\underline{i}\in I^{Sp}_{d,2n}$:
$$I_{X^A(\underline{i})}(X^C(\underline{i}))=\{f\in\Gamma_h(X^A(\underline{i}))|f(V)=0,\forall V\in X^C(\underline{i})\}.$$ Even though $X^C(\underline{i})$ is irreducible, we can  say only that $I_{X^A(\underline{i})}(X^C(\underline{i}))$ is the radical of the ideal generated by $\{E_{\underline{i'}}|\underline{i'}\in I_{d-2,2n}\}$. To solve this problem, we need more preparation.

Our goal here is the following theorem.
\begin{theo*}
  On a Grassmannian variety, for $\underline{i}\in I^{Sp}_{d,2n}$, the defining ideal of $X^C(\underline{i})$ in $X^A(\underline{i})$ is generated by $\{E_{\underline{i'}}|\underline{i'}\in I_{d-2,2n}\}$.

  On a flag variety, for $\underline{w}\in \Delta^{Sp}_{d,2n}$, the defining ideal of $X^C(\underline{w})$ in $X^A(\underline{w})$ is generated by $\cup_{d=2}^{n}\{E_{\underline{i'}}|\underline{i'}\in I_{d-2,2n}\}$.
\end{theo*}

Firstly, according to De Concini's work\cite{DeConcini79}, we need to study only the degree-1 graded $\Gamma^1_h(X^A(\underline{i})), \underline{i}\in I_{d,2n}$.

\begin{prop}\cite{DeConcini79}\label{linear}
  The defining ideal $$I_{X^A(\underline{i})}(X^C(\underline{i})), \underline{i}\in I^{Sp}_{d,2n}$$(resp. $I_{X^A(\underline{w})}(X^C(\underline{w})),\underline{w}\in \Delta^{Sp}_{d,2n}$) is generated by the restrictions of some linear (i.e., degree-1) homogeneous sections in $\Gamma_h^1(X^A(\underline{i}))$ on $X^C(\underline{i})$ (resp. by the restrictions of some linear homogeneous sections in $\Gamma_h^1(X^A(\underline{w}))$ on $X^C(\underline{w})$).
\end{prop}

The following (until Lemma \ref{n=2m,any}) is a further interpretation of this proposition, and the definitions introduced here will not appear again later. Specifically, De Concini provided an algorithm\cite[(2.2),(2.4)]{DeConcini79} that modulo some linear relations\cite[(1.8)]{DeConcini79}, which are identically zero on $Gr^C(d,2n)$ (resp. on $\mathrm{Sp}_{2n}/P^C_d$), every homogeneous section $f\in \Gamma_h(Gr(d,2n))$ (resp. $\Gamma_h(Fl_{2n}(1,2,\cdots,n)$) is equivalent to a linear combination of ``opposite symplectic standard tableaux". The symplectic standard tableaux are homogeneous sections constructed by De Concini\cite[(2.3)]{DeConcini79}. Here, we use the word ``opposite" because in his definition, De Concini used the lower triangular Borel, not the upper triangular one. By a conjugate action of an anti-diagonal matrix, we can connect the lower triangular and upper triangular cases, so through an automorphism of homogeneous coordinate rings that induces such action of the anti-diagonal matrix, we can obtain the isomorphic images of De Concini's symplectic standard tableaux: we call them the opposite symplectic standard tableaux.  It is not necessary to further explore this definition, we simply need it to play a role such that $\forall \underline{i}\in I^{Sp}_{d,2n}$(resp. $\underline{w}\in \Delta^{Sp}_{n,2n}$), the nonzero opposite symplectic standard tableaux on $X^A(\underline{i})$ (resp. $X^A(\underline{w})$) are linearly independent on $X^C(\underline{i})$ (resp. $X^C(\underline{w})$)
\footnote{In order to avoid introducing a large number of unnecessary symbols, we only clarify the language in De Concini\cite{DeConcini79}. Actually, not only the linear relations but also some quadratic relations\cite[(1.1)]{DeConcini79} were previously used. However, these quadratic relations are identically zero on any matrix, and thus on the whole Grassmannian/flag, and in the case we are interested in, i.e., the case that $\{j_1,j_2,\cdots,j_s\}\subset \{k_1,k_2,\cdots,k_{s'}\}$\cite[(1.1)]{DeConcini79},  $H=0$ in (1.1) and (1.1) are just Pl$\ddot{u}$cker relations. That implies the columns involved in algorithm\cite[(2.2),(2.4)]{DeConcini79} are always the same. Another point is that he worked on the variety of matrices. In the different languages of subspace and matrix, the Pl$\ddot{u}$cker coordinates correspond to minor functions on matrices. Moreover, (through an action of anti-diagonal matrix) De Concini's matrix Schubert variety is formed by a collection of special matrix presentations of our $V\in X^C(\underline{i})\subset Gr^C(d,2n)$ (or $V_n$, for $\underline{V}=(V_1,\cdots,V_n)$ a flag)(see also \cite[p.10]{seshadri_introduction_2016}). Additionally,  a polynomial in Pl$\ddot{u}$cker coordinates is zero on a Schubert variety if and only if its corresponding minor function is zero on all the matrix presentations. Thus, his language is parallel to ours.}\cite[(3.5)]{DeConcini79}. In other words, they have the following property.
\begin{prop}\cite[(3.5)]{DeConcini79}
  A linear combination of opposite symplectic standard tableaux is zero on $X^C(\underline{i}),\underline{i}\in I^{Sp}_{d,2n}$ (resp. $X^C(\underline{w}),\underline{w}\in \Delta^{Sp}_{n,2n}$) implies that this linear combination is zero on the whole $X^A(\underline{i})$ (resp. $X^A(\underline{w})$).
\end{prop}

For a homogeneous $f\in \Gamma_h(X^A(\underline{i}))$, it can be written as two parts $$\begin{array}{rl}f=&(f_I=\mbox{generated by linear sections that are identically zero on $Gr^C(d,2n)$})|_{X^A(\underline{i})}\\&+(f_{II}=\mbox{linear combination of opposite symplectic standard tableaux})|_{X^A(\underline{i})}.\end{array}$$ If moreover $f\in I_{X^A(\underline{i})}(X^C(\underline{i}))$ or equivalently $f(V)=0$ for any $V\in X^C(\underline{i})$, then $f_{II}|_{X^C(\underline{i})}=0$. By the above proposition,  we immediately obtain $f_{II}|_{X^A(\underline{i})}=0$. Thus, $f=f_I|_{X^A(\underline{i})}$, which implies $I_{X^A(\underline{i})}(X^C(\underline{i}))$ is generated by some restrictions of linear sections in $\Gamma^1_h(Gr(d,2n))$ on $X^A(\underline{i})$. A similar argument is valid for the Schubert varieties on the flag variety: $I_{X^A(\underline{w})}(X^C(\underline{w}))$ is generated by some restrictions of linear sections in $\Gamma^1_h(Fl_{2n}(1,2,\cdots,n))$ on $X^A(\underline{w}),\forall\underline{w}\in \Delta^{Sp}_{n,2n}$. That is, Proposition \ref{linear} holds.

Therefore, for the Schubert varieties in Grassmannian, in order to obtain the goal theorem mentioned above, it is sufficient to prove that every degree-1 homogeneous section in $I_{Gr(d,2n)}(Gr^C(d,2n))$ can be spanned by $\{E_{\underline{i'}}|\underline{i'}\in I_{d-2,2n}\}$: that is, $\Gamma^1_h(Gr(d,2n))\cap I_{Gr(d,2n)}(Gr^C(d,2n))=Span\{E_{\underline{i'}}|\underline{i'}\in I_{d-2,2n}\}.$ In addition, since the projection $pr_d$ is surjective for every $1\leq d\leq n$ and the degree-1 sections on flag varieties are exactly from the degree-1 sections on different Grassmannians, to prove the theorem above, we only need to work on Grassmannians.

Recall that we work on $char\ k= 0$.

\begin{lem}\label{n=2m,any}
For $n=2m$, an even integer, let $\underline{i},\underline{j}\in I^{Sp}_{n,2n}$, $\underline{i}=(i_1,i_2,\cdots,i_m,2n+1-i_m,\cdots,2n+1-i_2,2n+1-i_1)$ and $\underline{j}=(j_1,j_2,\cdots,j_m,2n+1-j_m,\cdots,2n+1-j_2,2n+1-j_1)$ with $\{i_1,i_2,\cdots,i_m\}\cup\{j_1,j_2,\cdots,j_m\}=\{1,2,\cdots,n\}$.
Then
$$\begin{array}{c}(-1)^{m-1} p_{\underline{i}}+p_{\underline{j}}\\=\sum\limits_{a=0}^{m-1}\sum\limits_{\mbox{\tiny$\begin{array}{c}
                                                                                                \underline{l}=(l_1,\cdots,l_{m-1},2n+1-l_{m-1},\cdots,2n+1-l_1)\in I^{Sp}_{n-2,2n}\\
                                                                                                \#(\underline{l}\cap \underline{j})=2a.
                                                                                               \end{array}$}}\frac{(-1)^a}{m\binom{m-1}{a}}E_{\underline{l}}\end{array}$$
\end{lem}
\begin{proof}
  Note that in such $E_{\underline{l}}$, the inversion numbers in front of the Pl$\ddot{u}$cker coordinates are all $(-1)^{m-1}$. Thus, this lemma can be deduced from a direct computation on the coefficient of every term $p_{(k_1,\cdots, k_m,2n+1-k_m,\cdots,2n+1-k_1)}$ on the right side with $\#(\{k_1,\cdots, k_m,2n+1-k_m,\cdots,2n+1-k_1\}\cap \underline{j})=b$: it is $$m\cdot \frac{(-1)^0}{m\binom{m-1}{0}}(-1)^{m-1}=(-1)^{m-1}$$ for $b=0$ ; $$m\cdot \frac{(-1)^{m-1}}{m\binom{m-1}{m-1}}(-1)^{m-1}=(-1)^{2(m-1)}=1$$ for $b=2m$; and $$(-1)^{m-1+b}(\frac{m-b}{m\binom{m-1}{b}}-\frac{b}{m\binom{m-1}{b-1}})=0$$ for any other $b$.
\end{proof}

The lemma above shows that for even integer $n$, if we partition $\{1,2,\cdots,n\}$ into two half parts $\{i_1,i_2,\cdots,i_m\}\cup\{j_1,j_2,\cdots,j_m\}$ and put $\underline{i}=(i_1,i_2,\cdots,i_m,2n+1-i_m,\cdots,2n+1-i_2,2n+1-i_1)$,$\underline{j}=(j_1,j_2,\cdots,j_m,2n+1-j_m,\cdots,2n+1-j_2,2n+1-j_1)$, then $p_{\underline{i}}\pm p_{\underline{j}}$ is a linear combination of $\{E_{\underline{i'}}|\underline{i'}\in I_{d-2,2n}\}$.

We require a simple result of the standard monomial theory on Grassmannians here.
\begin{theo}\cite[p.171]{lakshmibai_flag_2018},\cite[p.12]{seshadri_introduction_2016}\label{SMT basis}
  $\{p_{\underline{j}}|\underline{j}\in I_{d,2n}\}$ forms a basis for $\Gamma^1_h(Gr(d,2n))$. Additionally, $\{p_{\underline{j}}|\underline{i}\in I_{d,2n},\underline{j}\leq \underline{i}\}$ forms a basis for $\Gamma^1_h(X^A(\underline{i})),\forall \underline{i}\in I_{d,2n}$.
\end{theo}
  If $d=0$ or $1$, we say $E_{\underline{i'}}=0,\forall \underline{i'}\in 'I_{d-2,2n}'$ for convenience. For any $1\leq d\leq n, \underline{i}\in I_{d,2(n-1)}$, we naturally denote $\underline{i+1}=(i_1+1,i_2+1,\cdots,i_d+1)\in I_{d,2n}$.
\begin{theo}\label{main1}
  $\Gamma^1_h(Gr(d,2n))\cap I_{Gr(d,2n)}(Gr^C(d,2n))=Span\{E_{\underline{i'}}|\underline{i'}\in I_{d-2,2n}\}.$
\end{theo}

\begin{proof}
 Clearly, for $d=0,1,\forall n$. Let us use the induction on $d$ and $n$. For $d>1, n\geq d$, assume that $\Gamma^1_h(Gr(d',2n))\cap I_{Gr(d',2n')}(Gr^C(d',2n'))=Span\{E_{\underline{i''}}|\underline{i''}\in I_{d'-2,2n'}\}$ for any $d'\leq d, d'\leq n'\leq n, (d',n')\neq (d,n)$.
 Consider $\sum_{\underline{i}\in \Lambda} c_{\underline{i}}p_{\underline{i}}\in \Gamma^1_h(Gr(d,2n))\cap I_{Gr(d,2n)}(Gr^C(d,2n))$, where $c_{\underline{i}}\in k$ are coefficients and $\Lambda$ is an index set. Equivalently,
 $$\sum\limits_{\underline{i}\in \Lambda} c_{\underline{i}}p_{\underline{i}}|_{Gr^C(d,2n)}=0.$$
We want to show that such $\sum c_{\underline{i}}p_{\underline{i}}$ can be written as a linear combination of $\{E_{\underline{i'}}|\underline{i'}\in I_{d-2,2n}\}$.

 Put $2n\times 2n$ matrix $D_\lambda=diag(\lambda,1,1,\cdots,1,\lambda^{-1})$ for $\lambda\in k$. $D_\lambda$ naturally acts on $Gr(d,2n)$ and sends $Gr^C(d,2n)$ to $Gr^C(d,2n)$. It is induced by automorphism on $\Gamma_h(Gr(d,2n))$ sending $$\begin{array}{ll}
                                                p_{\underline{i}}\mapsto \lambda p_{\underline{i}}, &\mbox{if }1\in \underline{i},2n\not\in \underline{i}; \\
                                                p_{\underline{i}}\mapsto \lambda^{-1} p_{\underline{i}}, &\mbox{if }1\not \in \underline{i},2n\in \underline{i}; \\
                                                p_{\underline{i}}\mapsto p_{\underline{i}}, &\mbox{if }(1,2n\in \underline{i})\mbox{ or }(1,2n\not\in \underline{i}).
                                              \end{array}$$

 Additionally, this automorphism of ring maps $I_{Gr(d,2n)}(Gr^C(d,2n))$ to itself. Thus, the image
 $$(\lambda\sum\limits_{\underline{i}\in \Lambda_1}c_{\underline{i}}p_{\underline{i}}+\lambda^{-1}\sum\limits_{\underline{i}\in\Lambda_2}c_{\underline{i}}p_{\underline{i}}+\sum\limits_{\underline{i}\in\Lambda_3}c_{\underline{i}}p_{\underline{i}})|_{Gr^C(d,2n)}=0,$$

where $\Lambda_1=\{\underline{i}\in \Lambda|1\in \underline{i},2n\notin\underline{i}\}, \Lambda_2=\{\underline{i}\in \Lambda|1\notin\underline{i},2n\in \underline{i}\}, \Lambda_3=\{\underline{i}\in \Lambda|(1,2n\in \underline{i})\mbox{ or }(1,2n\notin\underline{i})\}.$ By the arbitrariness of $\lambda$, we have the following three parts
$$\sum\limits_{\underline{i}\in \Lambda_1}c_{\underline{i}}p_{\underline{i}}|_{Gr^C(d,2n)}=\sum\limits_{\underline{i}\in\Lambda_2}c_{\underline{i}}p_{\underline{i}}|_{Gr^C(d,2n)}=\sum\limits_{\underline{i}\in\Lambda_3}c_{\underline{i}}p_{\underline{i}}|_{Gr^C(d,2n)}=0.$$

Thus, we only need the proof for the three types $\Lambda=\Lambda_1$, $\Lambda=\Lambda_2$ and $\Lambda=\Lambda_3$.

\textbf{Type 1.} $\Lambda=\Lambda_1. \forall \underline{i}\in \Lambda, 1\in \underline{i}$ and $2n\notin \underline{i}$.

 Consider the embedding $\varphi_1:Gr(d-1,2(n-1))\to Gr(d,2n)$, which maps $V\in Gr(d-1,2(n-1))$ with a matrix presentation $M$ to $\varphi_1(V)\in Gr(d,2n)$ with matrix presentation $M_1$ as follows:
$$V=Span\{\mbox{columns of }M\}\mapsto \varphi(V)=Span\{\mbox{columns of } M_1=\left[\begin{matrix}
                       1 & 0 \\
                       0_{2(n-1)\times 1} & M \\
                       0 & 0
                     \end{matrix}\right]\}.$$

It is induced by $\varphi_1^*:\Gamma_h^1(Gr(d,2n))\to \Gamma^1_h(Gr(d-1,2(n-1)))$ sending $$p_{\underline{j+1}\cup\{1\}}\mapsto p_{\underline{j}},\forall \underline{j}\in I_{d-1,2(n-1)}$$ and having a kernel spanned by
$$\{p_{\underline{j}}|\underline{j}\in I_{d,2n}, 2n\in \underline{j}\mbox{ or } 1\notin \underline{j}\}.$$

For $V\in Gr^C(d-1,2(n-1))$, $\varphi_1(V)\in Gr^C(d,2n)$. Then, $$\varphi_1^*(\sum\limits_{\underline{i}\in \Lambda_1}c_{\underline{i}}p_{\underline{i}})|_{Gr^C(d-1,2(n-1))}=0,\varphi_1^*(\sum\limits_{\underline{i}\in \Lambda_1}c_{\underline{i}}p_{\underline{i}})\in I_{Gr(d-2,2(n-1))}(Gr^C(d-2,2(n-1))).$$ From our assumption on $Gr^C(d-1,2(n-1))$, we have
$$\varphi_1^*(\sum\limits_{\underline{i}\in \Lambda_1}c_{\underline{i}}p_{\underline{i}})=\sum\limits_{\underline{i''}\in I_{d-3,2(n-1)}}c_{\underline{i^{\prime\prime}}}E_{\underline{i''}}.$$

Note that one preimage of $\sum c_{\underline{i''}}E_{\underline{i''}}$ under $\varphi_1^*$ is $\sum c_{\underline{i''}}E_{\underline{i''+1}\cup\{1\}}$(or zero, if $d-1\leq 1$). So $$\sum\limits_{\underline{i}\in \Lambda_1}c_{\underline{i}}p_{\underline{i}}-\sum\limits_{\underline{i''}\in I_{d-3,2(n-1)}}c_{\underline{i''}}E_{\underline{i''}\cup\{1\}}\in \ker\varphi_1^*.$$

However, for every subscript $\underline{i}$ of $p_{\underline{i}}$ appearing in $\sum\limits_{\underline{i}\in \Lambda_1}c_{\underline{i}}p_{\underline{i}}-\sum\limits_{\underline{i''}\in I_{d-3,2(n-1)}}c_{\underline{i''}}E_{\underline{i''}\cup\{1\}}$, there must be $1\in \underline{i}, 2n\notin \underline{i}$. $\ker\varphi_1^*$ is spanned by $\{p_{\underline{j}},\underline{j}\in I_{d-1,2(n-1)}|2n\in \underline{j}\mbox{ or } 1\notin \underline{j}\}.$ Because of Theorem \ref{SMT basis}, there must be $$\sum\limits_{\underline{i}\in \Lambda_1}c_{\underline{i}}p_{\underline{i}}-\sum\limits_{\underline{i''}\in I_{d-3,2(n-1)}}c_{\underline{i''}}E_{\underline{i''}\cup\{1\}}=0$$ in $\Gamma^1_h(Gr(d,2n))$.

\textbf{Type 2.} $\Lambda=\Lambda_2.$ $\forall \underline{i}\in \Lambda, 1\notin\underline{i}$ and $2n\in\underline{i}$.

 Consider the embedding $Gr(d-1,2(n-1))\to Gr(d,2n)$ defined as
$$V=Span\{\mbox{columns of }M\}\mapsto \varphi(V)=Span\{\mbox{columns of } M_1=\left[\begin{matrix}
                       0 & 0 \\
                      M& 0_{2(n-1)\times 1} \\
                       0 & 1
                     \end{matrix}\right]\}.$$

It is induced by $\Gamma_h^1(Gr(d,2n))\to \Gamma^1_h(Gr(d-1,2(n-1)))$ sending $$p_{\underline{j+1}\cup\{2n\}}\mapsto p_{\underline{j}},\forall \underline{j}\in I_{d-1,2(n-1)}$$ and having a kernel spanned by
$$\{p_{\underline{j}}|\underline{j}\in I_{d,2n}, 1\in \underline{j}\mbox{ or } 2n\notin \underline{j}\}.$$ Then, similar to Type 1, we can express such a Type 2 linear section as a linear combination of $\{E_{\underline{i'}}|\underline{i'}\in I_{d-2,2n}\}$.

\textbf{Type 3.} $\Lambda=\Lambda_3.$ $\forall \underline{i}\in \Lambda,$ both of $1,2n\in \underline{i}$ or $1,2n\notin \underline{i}$.

We have two cases: $d<n$ or $d=n$.

If $d<n$, consider the embedding $\varphi_2: Gr(d,2(n-1))\to Gr(d,2n)$ defined as
$$V=Span\{\mbox{columns of } M\}\mapsto \varphi_2(V)=Span\{\mbox{columns of } \left[\begin{matrix}
                                                                                                      0_{1\times d} \\
                                                                                                      M \\
                                                                                                      0_{1\times d}
                                                                                                    \end{matrix}\right]\}.$$
It is induced by $\varphi^*_2:\Gamma_h^1(Gr(d,2n))\to \Gamma_h^1(Gr(d,2(n-1)))$ sending
$$p_{\underline{j+1}}\mapsto p_{\underline{j}},\forall \underline{j}\in I_{d,2(n-1)}$$ and having a kernel spanned by
$$\{p_{\underline{j}}|\underline{j}\in I_{d,2n}, 1\in \underline{j}\mbox{ or }2n\in \underline{j}\}.$$
Now, there must be $$\varphi_2^*(\sum\limits_{\underline{i}\in \Lambda}c_{\underline{i}}p_{\underline{i}})\in \Gamma^1_h(Gr(d,2(n-1)))\cap I_{Gr(d,2(n-1))}(Gr^C(d,2(n-1))).$$ By our assumption $$\varphi_2^*(\sum\limits_{\underline{i}\in \Lambda}c_{\underline{i}}p_{\underline{i}})=\sum\limits_{\underline{i''}\in I_{d-2,2(n-1)}}c_{\underline{i''}}E_{\underline{i''}}.$$

Similarly as in Type 1, one preimage of $\sum c_{\underline{i''}}E_{\underline{i''}}$ under $\varphi^*_2$ is $\sum c_{\underline{i''}}E_{\underline{i''+1}}$.

Thus, $$\sum\limits_{\underline{i}\in \Lambda}c_{\underline{i}}p_{\underline{i}}-\sum\limits_{\underline{i''}\in I_{d-2,2(n-1)}} c_{\underline{i''}}E_{\underline{i''+1}}\in \ker\varphi_2^*.$$

Here, $\ker\varphi_2^*$ is spanned by $\{p_{\underline{j}}|\underline{j}\in I_{d,2n}, 1\in\underline{j}\mbox{ or }2n\in\underline{j}\}$. Comparing the subscripts, we have

$$\sum\limits_{\underline{i}\in \Lambda}c_{\underline{i}}p_{\underline{i}}-\sum\limits_{\underline{i''}\in I_{d-2,2(n-1)}} c_{\underline{i''}}E_{\underline{i''+1}}=\sum\limits_{\underline{j}\in \Lambda_4}c_{\underline{j}}p_{\underline{j}},$$
where $\forall\underline{j}\in \Lambda_4\subset I_{d,2n}$, both of $1,2n\in \underline{j}$. Now, we must prove that $\sum\limits_{\underline{i}\in \Lambda_4}c_{\underline{j}}p_{\underline{j}}$ is a linear combination of $\{E_{\underline{i'}}|\underline{i}\in I_{d-2,2n}\}$. We leave this as a lemma.

In the case $d=n$ and for Type 3 $\sum\limits_{\underline{i}\in \Lambda}c_{\underline{i}}p_{\underline{i}}$, for any $1\leq t\leq n$, we have also
$$\sum\limits_{\tiny\mbox{$\begin{array}{c}
                              \underline{i}\in \Lambda\\t\in \underline{i}\\2n+1-t\notin\underline{i}\end{array}$}}c_{\underline{i}}p_{\underline{i}}|_{Gr^C(d,2n)}=\sum\limits_{\tiny\mbox{$\begin{array}{c}
                              \underline{i}\in \Lambda\\t\not\in \underline{i}\\2n+1-t\in\underline{i}\end{array}$}}c_{\underline{i}}p_{\underline{i}}|_{Gr^C(d,2n)}=0.$$ Therefore, they are linear combinations of $\{E_{\underline{i'}}|\underline{i'}\in I_{d-2,2n}\}$ through an argument similar to our proof for Types 1 and 2 above.

Therefore, after repeatedly excluding such terms, we have
$$\sum\limits_{\underline{i}\in \Lambda}c_{\underline{i}}p_{\underline{i}}-(\mbox{ a linear combination of $\{E_{\underline{i'}}|\underline{i'}\in I_{d-2,2n}\}$})=\sum\limits_{\underline{i}\in \Lambda_4'}c_{\underline{i}}p_{\underline{i}},$$
where for $\forall \underline{i}\in \Lambda_4', \forall 1\leq t\leq n$, $t,2n+1-t\in \Lambda_4'$ or $t,2n+1-t\notin \Lambda_4'$. If $d=n$ is odd, the proof is complete. If $d=n=2m$ is even, we still need to show $\sum_{\underline{i}\in \Lambda_4'}c_{\underline{i}}p_{\underline{i}}$ is a linear combination of $\{E_{\underline{i'}}|\underline{i'}\in I_{d-2,2n}\}$. Note that by subtracting some relations in Lemma \ref{n=2m,any}, which are linear combinations of $\{E_{\underline{i'}}|\underline{i'}\in I_{d-2,2n}\}$, we have $$\sum_{\underline{i}\in \Lambda_4'}c_{\underline{i}}p_{\underline{i}}-(\mbox{a linear combination of $\{E_{\underline{i'}}|\underline{i'}\in I_{d-2,2n}\}$})=\sum_{\underline{i}\in \Lambda_4}c_{\underline{i}}p_{\underline{i}},$$where $\forall \underline{i}\in\Lambda_4$, both  $1,2n\in \underline{i}$. Additionally, we have the condition $\sum_{\underline{i}\in \Lambda_4}c_{\underline{i}}p_{\underline{i}}|_{Gr^C(d,2n)}=0$. Thus, our proof is completed by the following lemma.
\end{proof}

Remark. For any $\underline{i}\in I_{d,2n}$, the embedding $\varphi_1$ gives an isomorphism $X^A(\underline{i})\to X^A(\underline{i+1}\cup\{1\})$, which also sends $X^C(\underline{i})\to X^C(\underline{i+1}\cup\{1\})$ if $\underline{i}\in I^{Sp}_{d,2n}$.

\begin{lem}
  Let $d>1, n\geq d$. If $$\Gamma^1_h(Gr(d',2n))\cap I_{Gr(d',2n')}(Gr^C(d',2n'))=Span\{E_{\underline{i''}}|\underline{i''}\in I_{d'-2,2n'}\}$$ for any $d'\leq d, d'\leq n'\leq n, (d',n')\neq (d,n)$ (the induction assumption), and $$\sum_{\underline{i}\in \Lambda}c_{\underline{i}}p_{\underline{i}}|_{Gr^C(d,2n)}=0,$$ where $\forall\underline{i}\in \Lambda,$ both of $1,2n\in\underline{i}$, then $\sum_{\underline{i}\in \Lambda}c_{\underline{i}}p_{\underline{i}}$ is a linear combination of $\{E_{\underline{i}}|\underline{i}\in I_{d-2,2n}\}$.
\end{lem}
\begin{proof}
  Consider the embedding $\varphi_3:Gr(d-2,2(n-1))\to Gr(d,2n)$:
  $$\begin{array}{c}V=Span\{\mbox{columns of }M\}\mapsto\\ \varphi_3(V)=Span\{\mbox{columns of }\left[\begin{matrix}
                                                                                        1 & 0_{1\times(d-2)} & 0 \\
                                                                                        0_{2(n-1)\times1} & M & 0_{2(n-1)\times1} \\
                                                                                        0 & 0_{1\times(d-2)} & 1
                                                                                      \end{matrix}\right]\}.\end{array}$$

It is induced by $\varphi_3^*:\Gamma_h^1(Gr(d,2n))\to \Gamma_h^1(Gr(d-2,2(n-1)))$ sending
$$\begin{array}{ll}p_{\underline{i'}\cup{1,2n}}\mapsto p_{\underline{i'}},&\forall\underline{i'}\in I_{d-2,2n}\\
p_{\underline{i}}\mapsto 0,&\forall \underline{i}\in I_{d,2n}\mbox{ s.t. }1\not\in\underline{i}\mbox{ or }2n\notin\underline{i}.\end{array}$$

We should be careful that $\varphi_3(Gr^C(d-2,2(n-1)))\neq Gr^C(d,2n)$. However, for any $V\in Gr^C(d-2,2(n-1))$ with matrix presentation $M_{2(n-1)\times (d-2)}$, since $d-2\leq n-2=(n-1)-1$, we can extend $V$ to a $(n-2)$-dimensional subspace $V'$ of $k^{2(n-1)}$ with $V'\perp JV'$. The orthogonal complement of $JV'$ in $k^{2(n-1)}$ is of dimension $2(n-1)-(n-2)=n=(n-1)+1$, so we can find two linearly independent vectors $m_1,m_2\in k^{2(n-1)}$ such that (here $J$ is $(d-2)\times(d-2)$ standard symplectic matrix)$$m_1^TJM=m_2^TJM=0, m_1^TJm_2\neq 0.$$
We can assume $m_1^TJm_2=1$. Then,
$$\tilde{V}=Span\{\mbox{columns of }\left[\begin{matrix}
                                                                                        1 & 0 & 0 \\
                                                                                        m_1 & M & m_2 \\
                                                                                        0 & 0 & 1
                                                                                      \end{matrix}\right]\}\in Gr^C(d,2n).$$

Clearly, using a series of elementary row transformations ($t$-th row)-$\lambda$(1-st row), ($t$-th row)-$\lambda$(2n-th row), $2\leq t\leq 2n-1,
\lambda\in k$, we can obtain
$$\left[\begin{matrix}
                                                                                        1 & 0 & 0 \\
                                                                                        m& M & m_2 \\
                                                                                        0 & 0 & 1
                                                                                      \end{matrix}\right]\longrightarrow\left[\begin{matrix}
                                                                                        1 & 0 & 0 \\
                                                                                        0 & M & 0 \\
                                                                                        0 & 0 & 1
                                                                                      \end{matrix}\right].$$

Furthermore, recall that in $$\sum_{\underline{i}\in \Lambda}c_{\underline{i}}p_{\underline{i}},$$ $\forall\underline{i}\in \Lambda,$ both  $1,2n\in\underline{i}$. Thus,
$$\sum_{\underline{i}\in \Lambda}c_{\underline{i}}p_{\underline{i}}(\varphi_3(V))=\sum_{\underline{i}\in \Lambda}c_{\underline{i}}p_{\underline{i}}(\tilde{V})=0.$$

We obtain $\varphi^*_3(\sum_{\underline{i}\in \Lambda}c_{\underline{i}}p_{\underline{i}})|_{Gr^C(d-2,2(n-1))}=0;$ then, by the induction assumption
$$\varphi^*_3(\sum_{\underline{i}\in \Lambda}c_{\underline{i}}p_{\underline{i}})=\sum\limits_{\underline{i''}\in I_{d-2,2(n-1)}} c_{\underline{i''}}E_{\underline{i''}}.$$

One preimage of $\sum\limits_{\underline{i''}\in I_{d-2,2(n-1)}} c_{\underline{i''}}E_{\underline{i''}}$ under $\varphi^*_3$ is $\sum\limits_{\underline{i''}\in I_{d-2,2(n-1)}} c_{\underline{i''}}E_{\underline{i''+1}\cup\{1,2n\}}$. Additionally, $\ker\varphi_3^*$ is spanned by $$\{p_{\underline{i}},\underline{i}\in I_{d,2n}|1\not\in\underline{i}\mbox{ or }2n\notin\underline{i}\}.$$ Comparing the subscripts, we have $$\sum_{\underline{i}\in \Lambda}c_{\underline{i}}p_{\underline{i}}=\sum\limits_{\underline{i''}\in I_{d-2,2(n-1)}} c_{\underline{i''}}E_{\underline{i''+1}\cup\{1,2n\}}.$$
\end{proof}

Finally, combining Propositions \ref{local relation} and \ref{flag E}, we obtain the main theorem.
\begin{theo}\label{main2}
  On a Grassmannian variety, for $\underline{i}\in I^{Sp}_{d,2n}$, the defining ideal of $X^C(\underline{i})$ in $X^A(\underline{i})$ is generated by $\{E_{\underline{i'}}|\underline{i'}\in I_{d-2,2n}\}$. In other words, $X^C(\underline{i})$ is a scheme-theoretical intersection of $X^A(\underline{i})$ with hyperplanes $\{E_{\underline{i'}}|\underline{i'}\in I_{d-2,2n}\}$. If $\underline{j}\in I^{Sp}_{d,2n}$ such that $\underline{j}\leq \underline{i}$, then locally, the defining ideal of $X^C(\underline{i})\cap A_{\underline{
  j}}$ in $X^A(\underline{i})\cap A_{\underline{j}}$ is generated by $\{\frac{E_{\underline{i'}}}{p_{\underline{j}}}|\underline{i'}\in I_{d-2,2n},\underline{i'}\subset\underline{i}\}$.

  On a flag variety, for $\underline{w}\in \Delta^{Sp}_{d,2n}$, the defining ideal of $X^C(\underline{w})$ in $X^A(\underline{w})$ is generated by $\cup_{d=2}^{n}\{E_{\underline{i'}}|\underline{i'}\in I_{d-2,2n}\}$. In other words, $X^C(\underline{w})$ is a scheme-theoretical intersection of $X^A(\underline{w})$ with hyperplanes $\cup_{d=2}^{n}\{E_{\underline{i'}}|\underline{i'}\in I_{d-2,2n}\}$. If $\underline{u}\in \Delta^{Sp}_{d,2n}$ such that $\underline{u}\leq \underline{w}$, then locally, the defining ideal of $X^C(\underline{w})\cap O_{\underline{
  u}}$ in $X^A(\underline{w})\cap O_{\underline{u}}$ is generated by $\{\frac{E_{\underline{i'}}}{p_{\underline{u}^{(n)}}}|\underline{i'}\in I_{n-2,2n},\underline{i'}\subset\underline{u}^{(n)}\}$.
\end{theo}
\section{Number of required defining equations in Schubert varieties}

In this section, we discuss the number of defining equations required to obtain $X^C(\underline{i})$ from $X^A(\underline{i})\subset Gr(d,2n)$. By Lemma \ref{ideal of type Aschubert}, we can see that many of $\{E_{\underline{i'}}|\underline{i'}\in I_{d-2,2n}\}$ will be identically zero on $X^A(\underline{i})$. Moreover, from Theorem \ref{SMT basis}, we have $E_{\underline{i'}}|_{X^A(\underline{i})}=0$ if and only if $\underline{i'}\cup\{t,2n+1-t\}\leq \underline{i}$ for all $1\leq t\leq n,t\notin \underline{i'}$. Our motivation is to exactly count the number of such equations. Of course not only the situation of Grassmannian but also one of flag variety should be discussed, but from Lemma \ref{ideal of type Aschubert} and the following Lemma \ref{dim X}, it is sufficient to work on Grassmannian cases. We do this locally, that is, on every affine open $A_{\underline{j}}$.

Now, for $\underline{j}\leq\underline{i}\in I^{Sp}_{d,2n}$, let $N_{\underline{j},\underline{i}}$ be the number
$$\#\{1\leq s<t\leq d |E_{\underline{j}}\backslash\{j_s,j_t\}|_{X^A(\underline{i})}\neq 0\}.$$

Recall that $W(\mathrm{Sp}_{2n})\subset W(\mathrm{SL}_{2n})=S_{2n}, W(\mathrm{Sp}_{2n})=\{w=(w_1w_2\cdots w_{2n})\in S_{2n}|w_t+w_{2n+1-t}=2n+1,\forall 1\leq t\leq n\}.$  For $w\in W(\mathrm{SL}_{2n})=S_{2n}$ written on one-line, the length is $l^A(w)=\tau(w)$, i.e., the inversion number, and for $w\in W(\mathrm{Sp}_{2n}), l^C(w)=(\tau(w)+m)/2$, where $m=\#\{1\leq t\leq d|w_t>n\}$. We have the identifications $I_{d,2n}=W(\mathrm{SL}_{2n})_{P^A_d}^{\min}$ and $I^{Sp}_{d,2n}=W(\mathrm{Sp}_{2n})_{P^C_d}^{\min}$. However, these two identifications work in different ways; they are, respectively,
$$\begin{array}{c}\underline{i}=(i_1,i_2,\cdots,i_d)\mapsto \\\underline{i}^A=(i_1,i_2,\cdots, i_d,i_{d+1},\cdots, i_{2n})\in W(\mathrm{SL}_{2n}),\\\mbox{i.e. } i_{d+1}<\cdots<i_{2n}\end{array}$$and$$\begin{array}{c}\underline{i}=(i_1,i_2,\cdots,i_d)\mapsto \\
\underline{i}^C=(i_1,i_2,\cdots, i_d,i_{d+1},\cdots,i_n,i_{n+1},\cdots,i_{2n})\in W(\mathrm{Sp}_{2n}),\\\mbox{i.e. }i_{d+1}<\cdots<i_{n},\forall 1\leq t\leq n,i_t+i_{2n+1-t}=2n+1.\end{array}$$

If $d=n$, then $\underline{i}^A=\underline{i}^C$, but in general, they are distinct.

We denote $l(\underline{i},A)=l^A(\underline{i}^A)$ and $l(\underline{i},C)=l^C(\underline{i}^C),\forall \underline{i}\in I^{Sp}_{d,2n}$.

\begin{lem}\cite[p.119]{seshadri_introduction_2016}\label{dim X}
  For $\forall \underline{i}\in I^{Sp}_{d,2n},$ $$\dim X^A(\underline{i})=l(\underline{i},A),\dim X^C(\underline{i})=l(\underline{i},C).$$ For $\underline{w}=(w_1,w_2,\cdots,w_n)\in \Delta^{Sp}_{n,2n}$, $$\dim X^A(\underline{w})=l(\underline{w}^{(n)},A)+\tau(w_1,w_2,\cdots,w_n),$$$$\dim X^C(\underline{w})=l(\underline{w}^{(n)},C)+\tau(w_1,w_2,\cdots,w_n).$$
\end{lem}

\begin{examp}
  Take $d=3,n=4$ and $\underline{i}=(137)\in I^{Sp}_{3,8}$. $$\underline{i}^A=(13724568)$$and$$\underline{i}^C=(13745268).$$ Thus, $\dim X^A((137))=5,\dim X^C((137))=(7+1)/2=4$. Now, let us consider $N_{(137),(137)}$ and $N_{(123),(137)}$.
  $$\begin{array}{c}E_{1}=\pm p_{127}\pm p_{136}\pm p_{145},\\
  E_{3}=\pm p_{138}\pm p_{237}\pm p_{345}\\
  E_{7}=\pm p_{178}\pm p_{367}\pm p_{457}\end{array}$$ By Lemma \ref{ideal of type Aschubert} and Theorem \ref{SMT basis}, we see that $N_{(137),(137)}=1$;
  $$\begin{array}{c}E_{1}=\pm p_{127}\pm p_{136}\pm p_{145},\\
  E_{2}=\pm p_{128}\pm p_{236}\pm p_{245}\\
  E_{3}=\pm p_{138}\pm p_{237}\pm p_{345}\end{array}$$ So $N_{(123),(137)}=2$. Note that $N_{(137),(137)}=\dim X^A((137))-\dim X^C((137))$, but $N_{(123),(137)}\neq \dim X^A((137))-\dim X^C((137))$.
\end{examp}

\begin{prop}\label{complete intersection in Ai}
  $X^C(\underline{i})=X^A(\underline{i})\cap Gr^C(d,2n)$, and in affine open $A_{\underline{i}}$, $N_{\underline{i},\underline{i}}=l(\underline{i},A)-l(\underline{i},C)$ . Particularly, $X^C(\underline{i})\cap A_{\underline{i}}$ is a complete intersection of $X^A(\underline{i})\cap A_{\underline{i}}$.
\end{prop}
The proof of Proposition \ref{complete intersection in Ai} relies on the following idea of exclusion.
\begin{lem}
  For $\underline{i}\in I_{d,2n}, 1\leq s< t\leq d$, if $i_s+i_t\leq 2n$, then $E_{\underline{i}\backslash \{i_s,i_t\}}|_{X^A(\underline{i})}=0$.
\end{lem}
\begin{proof}
  $E_{\underline{i}\backslash \{i_s,i_t\}}=\sum_{r=1}^{n}\pm p_{\underline{i}\backslash \{i_s,i_t\}\cup \{r,2n+1-r\}}$, and for every $\underline{j}=\underline{i}\backslash \{i_s,i_t\}\cup \{r,2n+1-r\},$ there must be $\underline{j}\not\leq \underline{i}$ because $$\sum_{k=1}^{d}j_k=\sum_{k=1}^{d}i_k-i_s-i_t+r+(2n+1-r)\geq \sum_{k=1}^{d}i_k+1.$$ By Lemma \ref{ideal of type Aschubert}, $E_{\underline{i}\backslash \{i_s,i_t\}}|_{X^A(\underline{i})}=0$ holds.
\end{proof}

Then, the number of equations $\#\{E_{\underline{i}\backslash \{i_s,i_t\}}, 1\leq s<t\leq d|E_{\underline{i}\backslash \{i_s,i_t\}}|_{X^A(\underline{i})}\neq 0\}$ is at most $\#\{1\leq s<t\leq d|i_s+i_t>2n\}.$ Next, we show that this number is exactly $l(\underline{i},A)-l(\underline{i},C),\forall \underline{i}\in I^{Sp}_{d,2n}$.

\begin{prop}\label{number in Ai}
  For $\underline{i}\in I^{Sp}_{d,2n}$, $\#\{1\leq s<t\leq d|i_s+i_t>2n\}=l(\underline{i},A)-l(\underline{i},C)$.
\end{prop}

\begin{proof}
  Fix $m=\#\{1\leq k\leq d|i_k>n\}$, use induction on $\sum_{k=1}^{d} i_k$.

  Note that for a fixed $m$, the minimal possible value of $\sum_{k=1}^{d} i_k$ is $1+2+\cdots+(d-m)+(n+1)+\cdots+(n+m)$. At this time, $$\underline{i}=(1,2,\cdots,d-m,n+1,\cdots,n+m).$$ Its shortest representative in $W(\mathrm{Sp}_{2n})$ is
  $$\underline{i}^C=(I,II,III,IV,V,VI).$$
  We partition this expression into six parts: $$\begin{array}{l}
                                              I=(1,2,\cdots,d-m),\\II=(n+1,n+2,\cdots,n+m),\\III=(d-m+1,d-m+2,\cdots,n-m),\\IV= (n-m+1,\cdots,2n+m-d),\\V=(n-m+1,n-m+2,\cdots,n),\\VI=(2n+m-d+1,2n+m-d+2,\cdots,2n). \end{array}$$

Then, $$\tau(\underline{i}^C)=\tau(II,III)+\tau(II,V)+\tau(IV,V)$$ and $$\tau(\underline{i}^A)=\tau(II,III)+\tau(II,V).$$

Thus, $$\begin{array}{rl}l(\underline{i},A)-l(\underline{i},C)&=\frac{\tau(II,III)+\tau(II,V)-\tau(IV,V)-m}{2}\\&=\frac{m(n-d)+m^2-m(n-d)-m}{2}\\&=\frac{m^2-m}{2}.\end{array}$$

Additionally, $\#\{1\leq s<t\leq d|i_s+i_t>2n\}=\frac{m(m-1)}{2}=l(\underline{i},A)-l(\underline{i},C)$. Now, for the minimal $\sum_{k=1}^{d} i_k$, our proposition is proved.

If $\sum_{k=1}^{d} i_k$ is greater than the minimal case, there must be some $i_k\neq 1, n+1$ such that $i_k-1\not\in \underline{i}$;  we can choose $i_r$ to be the minimal such value. Then,
\begin{itemize}
  \item[case 1.]$2n+1-(i_r-1)\not\in \underline{i}$.

  Let $\underline{j}=(\underline{i}\backslash i_r)\cup(i_r-1)\in I^{Sp}_{d,2n}$. Then, $$l(\underline{j},A)=l(\underline{i},A)-1,l(\underline{j},C)=l(\underline{i},C)-1.$$
  Additionally, $$\#\{1\leq s<t\leq d|j_s+j_t>2n\}=\#\{1\leq s<t\leq d|i_s+i_t>2n\},$$because for those $i_k$ such that $i_k+i_r>2n$, we have $i_k+i_r\neq 2n+1$, $i_k+i_r>2n+1$, which implies $i_k+(i_r-1)>2n$ as well. Note that $\sum_{k=1}^{d} j_k=\sum_{k=1}^{d} i_k-1$, so we can use induction assuming that $l(\underline{j},A)-l(\underline{j},C)=\#\{1\leq s<t\leq d|j_s+j_t>2n\}$.

  Therefore, $l(\underline{i},A)-l(\underline{i},C)=\#\{1\leq s<t\leq d|i_s+i_t>2n\}$ holds.
  \item[case 2.]$2n+1-(i_r-1)\in\underline{i}$.

  Let $\underline{j}=(\underline{i}\backslash\{i_r,2n+1-i_r\})\cup \{i_r-1,2n+1-i_r\}$. Then, $$l(\underline{j},A)=l(\underline{i},A)-2,l(\underline{j},C)=l(\underline{i},C)-1.$$ And $$\#\{1\leq s<t\leq d|j_s+j_t>2n\}=\#\{1\leq s<t\leq d|i_s+i_t>2n\},$$since the only removed such pair is $$i_r,2n+1-i_r,$$and the new pair $(i_r-1)+(2n+1-i_r)=2n\not> 2n$.

  $\sum_{k=1}^{d} j_k=\sum_{k=1}^{d} i_k-2$, by means of the induction assumption, we can obtain $l(\underline{j},A)-l(\underline{j},C)=\#\{1\leq s<t\leq d|j_s+j_t>2n\}$. Therefore, $l(\underline{i},A)-l(\underline{i},C)=\#\{1\leq s<t\leq d|i_s+i_t>2n\}$.
\end{itemize}
The proof is completed.
\end{proof}

Thus, Proposition \ref{complete intersection in Ai} is also proved.

\begin{lem}\label{number in affines lemma}
  For $\underline{i}\in I^{Sp}_{d,2n}$ and $\underline{i'},\underline{j'}\in I^{Sp}_{d-2,2n}$ such that $\underline{j'}\leq \underline{i'}$, if $E_{\underline{i'}}|_{X^A(\underline{i})}\neq 0$, then $E_{\underline{j'}}|_{X^A(\underline{i})}\neq 0$.
\end{lem}
\begin{proof}
  Without loss of generality, we can assume $l(\underline{j'},C)=l(\underline{i'},C)-1$, i.e. $\underline{j'}=\underline{i'}\backslash\{s\}\cup\{s-1\}$ or $\underline{j'}=\underline{i'}\backslash\{2n+2-s\}\cup\{2n+1-s\}$ or $\underline{j'}=\underline{i'}\backslash\{s,2n+2-s\}\cup\{s-1,2n+1-s\}$ for some $2\leq s\leq n$.

  $E_{\underline{i'}}|_{X^A(\underline{i})}\neq 0$ implies that there exist $1\leq t\leq n$ satisfying $t,2n+1-t\notin\underline{i'}$ and $\underline{i'}\cup\{t,2n+1-t\}\leq \underline{i}$. If $t\neq s-1$, there must be $\underline{j'}\cap\{t,2n+1-t\}\neq \empty$. Furthermore, $t$ (also, $2n+1-t$) appears at the same position in $\underline{j'}\cup\{t,2n+1-t\}$ and in $\underline{i'}\cup\{t,2n+1-t\}$. Easily $$\underline{j'}\cup\{t,2n+1-t\}\leq\underline{i'}\cup\{t,2n+1-t\}\leq \underline{i}.$$ Then, $E_{\underline{j'}}|_{X^A(\underline{i})}\neq 0$.

  If $t=s-1$, then there is only the possibility $\underline{j'}=\underline{i'}\backslash\{s\}\cup\{s-1\}$ or $\underline{j'}=\underline{i'}\backslash\{2n+2-s\}\cup\{2n+1-s\}$. At this time, by means of a straightforward comparison of
  $$\begin{array}{ll}
      \underline{j'}\cup\{t+1,2n-t\}, & \mbox{ if }\underline{j'}=\underline{i'}\backslash\{s\}\cup\{s-1\};  \\
     \underline{j'}\cup\{t-1,2n+2-t\}, & \mbox{ if }\underline{j'}=\underline{i'}\backslash\{2n+2-s\}\cup\{2n+1-s\}
    \end{array}$$
  with $\underline{i'}\cup\{t,2n+1-t\}$, we can see there must be $E_{\underline{j'}}|_{X^A(\underline{i})}\neq 0$.
  \end{proof}

\begin{coro}\label{number in affines}
  Let $\underline{i}\in I^{Sp}_{d,2n}$. If $\underline{l}\leq \underline{j}\in I^{Sp}_{d,2n}$, then the number
   $$N_{\underline{l},\underline{i}}\geq N_{\underline{j},\underline{i}}.$$ Specifically, $N_{\underline{i},\underline{i}}=\#\{1\leq s<t\leq d|i_s+i_t>2n\}=l(\underline{i},A)-l(\underline{i},C)$ and for all $\underline{j}\leq \underline{i}$,
   $$N_{id,\underline{i}}\geq N_{\underline{j},\underline{i}}.$$

\end{coro}

\begin{prop}
  For $\underline{i}\in I^{Sp}_{d,2n}$ (resp. $\underline{w}\in \Delta^{Sp}_{n,2n}$), $X^C(\underline{i})$ (resp. $X^C(\underline{w})$) is a local complete intersection in $X^A(\underline{i})$ (resp. $X^A(\underline{w})$) if and only if $$N_{id,\underline{i}}=N_{\underline{i},\underline{i}}$$ (resp. $N_{id,\underline{w}^{(n)}}=N_{\underline{w}^{(n)},\underline{w}^{(n)}}$).
\end{prop}
\begin{proof}
  We  prove this statement in only the Grassmannian case. The Bruhat decomposition of symplectic $Gr^C(d,2n)=\cup_{\underline{j}}\in I^{Sp}_{d,2n} B^C.e_{\underline{j}}$ implies that $Gr^C(d,2n)$ can be covered by affine open $A_{\underline{j}}\cap Gr^C(d,2n), \underline{j}\in I^{Sp}_{d,2n}$. Thus, by Proposition \ref{number in Ai} and Corollary \ref{number in affines}, the sufficiency is obtained. To see the necessity,
  if $$N_{id,\underline{i}}>N_{\underline{i},\underline{i}},$$ let $m_{id}$ be the maximal ideal of the local ring of $e_{id}$ in $X^A(\underline{i})\cap A_{id}$. Then, the residues of
  $$\{\frac{E_{id\backslash\{s,t\}}}{p_{id}},1\leq s<t\leq d|E_{id\backslash\{s,t\}}|_{X^A(\underline{i})}\neq 0\}$$ are linearly independent in $m_{id}/m_{id}^2$. That implies that $m_{id}$ cannot be generated by less than $N_{id,\underline{i}}$ elements. Thus, $X^C(\underline{i})$ cannot be a complete intersection of $X^A(\underline{i})$ in any open neighborhood.
\end{proof}

\begin{prop}\label{specific number}
  Let $\underline{i}\in I^{Sp}_{d,2n}$ with $N_{\underline{i},\underline{i}}>0$.  If $r_1=\min\{1\leq a\leq d| i_a\geq a+1\}, r_2=\min\{1\leq a\leq d|i_a\geq a+2\}, q=2n+1-i_d$, then the number $$N_{id,\underline{i}}=\binom{d-(r_1-1)}{2}-(\min\{r_2,q\}-r_1).$$
\end{prop}

\begin{proof}
  By the remark after Theorem \ref{main1}, we  need to prove this statement only for $r_1=1$, i.e., $i_1\geq 2$. Denote $l=\min\{r_2,q\}$.

  First, let us consider $l=1$, i.e., $i_1\geq 3$ or $i_d=2n$.

  In the case $i_1\geq 3$, since $N_{\underline{i},\underline{i}}>0$, for some $1\leq s<t\leq d$, we have $E_{\underline{i}\backslash\{i_s,i_t\}}|_{X^A(\underline{i}}\neq 0$. There must be $\underline{i}\backslash\{i_s,i_t\}\geq (3,4,\cdots,d)$, where $(3,4,\cdots,d)$ is the unique maximal (under Bruhat order) element among all the $id\backslash\{s',t'\}, 1\leq s'<t'\leq d$. Thus, for any $id\backslash\{s',t'\}, 1\leq s'<t'\leq d$, $E_{id\backslash\{s',t'\}}|_{X^A(\underline{i})}\neq 0$ (Corollary \ref{number in affines}). Then, $N_{\underline{i},\underline{i}}\geq \binom{d}{2}$, so it must be $\binom{d}{2}$.

  In the case $i_d=2n$, there must be $E_{\underline{i}\backslash\{i_1,i_d\}}|_{X^A(\underline{i}}\neq 0$. $i_2>i_1\geq 2$, so $\underline{i}\backslash\{i_1,i_d\}\geq (3,4,\cdots,d-2)$. Then, a similar reasoning reveals $N_{\underline{i},\underline{i}}=\binom{d}{2}$ as well.

  Now, let $l>1$. Then $i_1=2,i_2=3,\cdots,i_{l-1}=l$ and $i_d\neq 2n$. Consequently, there must be $i_d<2n+1-l$, $l<2n+1-i_d=q$; therefore, $l=r_2$, $i_l\geq l+2$.

  Since $N_{\underline{i},\underline{i}}>0$ and $i_1=2,i_2=3,\cdots,i_{l-1}=l,i_d< 2n+1-l$, for some $l\leq s<t\leq d$, we have $E_{\underline{i}\backslash\{i_s,i_t\}}|_{X^A(\underline{i})}\neq 0$. There must be $\underline{i}\backslash\{i_s,i_t\}\geq (2,3,\cdots,l,l+2,\cdots,d)$, where $(2,3,\cdots,l,l+2,\cdots,d)$ is the unique maximal (under Bruhat order) element in
   $$\{\underline{i'}=id\backslash\{s',t'\},1\leq s'<t'\leq d|i'_{l-1}\leq l\}.$$ This set is exactly $$\{\underline{i'}=id\backslash\{s',t'\},1\leq s'<t'\leq d\}\backslash\{id\backslash\{1,k\}|2\leq k\leq l\}.$$ By Corollary \ref{number in affines}, for any $\underline{i'}=id\backslash\{s',t'\}$ with $i'_{l-1}\leq l$, we have $E_{id\backslash\{s',t'\}}|_{X^A(\underline{i})}\neq 0$. Additionally, if $\underline{i'}=id\backslash\{1,k\},2\leq k\leq l$, we can see that $\underline{i'}\cup\{a,2n+1-a\}\not\leq \underline{i}$ for any $1\leq a\leq n, a\notin\underline{i'}$, which implies $E_{\underline{i'}}|_{X^A(\underline{i})}=0$.
   Then, $N_{id,\underline{i}}=\binom{d}{2}-(l-1)$.

   By replacing $d$, $r_2$, and $q$ by $d-(r_1-1)$, $r_2-(r_1-1)$, and $q-(r_1-1)$, we obtain the general formula.
\end{proof}

Remark. A special situation is $d=2$. If $N_{\underline{i},\underline{i}}>0$, then there must be $l=1$.
\begin{theo}\label{main3}
  For $\underline{i}\in I^{Sp}_{d,2n}$ and $r_1=\min\{1\leq a\leq d| i_a\geq a+1\}$, the symplectic Schubert variety $X^C(\underline{i})$ is a local complete intersection in $X^A(\underline{i})$ if and only if $i_s+i_t>2n,\forall r_1\leq s<t\leq d$. At this time, if $X^A(\underline{i})$ is smooth, then $X^C(\underline{i})$ is a local complete intersection (intrinsically).
\end{theo}
\begin{proof}
  Let $r_1,r_2,q,l$ be as in the proposition above. We only discuss the case for $N_{\underline{i},\underline{i}}>0$ and $r_1=1$. If $l> 1$ (then $d>2$), we have $i_1=2,i_2=3,\cdots,i_{l-1}=l,i_d<2n+1-l.$ Thus, $N_{\underline{i},\underline{i}}=\#\{1\leq s<t\leq d|i_s+i_t>2n\}\leq \binom{d-(l-1)}{2}$. However, $\binom{d}{2}-(l-1)\leq \binom{d-(l-1)}{2}$ only occurs when $d\leq 2$. It follows that $X^C(\underline{i})$ cannot be a local complete intersection of $X^A(\underline{i})$ for $l>1$.

  Thus, there must be $l=1$. At this time, $N_{id,\underline{i}}=\binom{d}{2}$. Therefore, if we want $N_{id,\underline{i}}=N_{\underline{i},\underline{i}}$, it is equivalent to $i_s+i_t>2n,\forall 1\leq s<t\leq d$. In general (not necessarily $r_1=1$), it is $i_s+i_t>2n,\forall r_1\leq s<t\leq d$.
\end{proof}

\begin{coro}
  For $\underline{w}\in \Delta^{Sp}_{n,2n}$ and $r_1=\min\{1\leq a\leq n|a\notin\underline{w}\}$, the symplectic Schubert variety $X^C(\underline{i})$ is a local complete intersection in $X^A(\underline{i})$ if and only if $\underline{w}^{(n)}=(1,2,\cdots,r_1-1,n,n+2,\cdots,2n+1-r_1)$. Moreover, it is equivalent to stating that the projection image $pr_n(X^A(\underline{w}))=X^A(\underline{w}^{(n)})\subset Gr(n,2n)$ is either a smooth Schubert variety or of codimension 1 in a smooth Schubert variety in $Gr(n,2n)$. At this time, if  $X^A(\underline{w})$ is smooth, then $X^C(\underline{w})$ is a local complete intersection (intrinsically).
\end{coro}
 \begin{proof}See \cite[(5.3)]{LAKSHMIBAI1990179} for the involved criteria for smoothness of Schubert varieties in $Gr(n,2n)$. Clearly, $pr_n(X^A(\underline{w}))=X^A(\underline{w}^{(n)})$ satisfies the condition in Theorem \ref{main3} if and only if it is smooth or of codimension 1 in a smooth Schubert variety.\end{proof}
\section{Symplectic conditions on tangent space}
Except for the well-studied situation of flag varieties, existing research on the singularities of Schubert varieties concentrates on the case where this Schubert variety lies in a minuscule or cominuscule $G/P$, where $G$ is a classical linear algebraic group and $P$ is the minuscule or cominuscule maximal parabolic subgroup of $G$. In $\mathrm{SL}_{2n}$, every maximal parabolic subgroup $P^A_{d},1\leq d\leq 2n$ is minuscule and cominuscule\cite{LAKSHMIBAI1990179}. However, in $\mathrm{Sp}_{2n}$, there are only $P^C_1$ and $P^C_n$. This leaves a gap in the research on $Gr^C(d,2n),1<d<n$. For $Gr^C(d,2n)$ (and even $Fl^C_{2n}(1,2,\cdots,n)$), there is an interesting results stating that $X^C(\underline{i})$ (resp. $X^C(\underline{w})$ must be smooth if $X^A(\underline{i})$ is already smooth \cite{LAKSHMIBAI1990179}(resp. its corresponding Schubert variety in $Fl_{2n}(1,2,\cdots,2n)$ is smooth\cite{LAKSHMIBAI1997332}). In this section, we compute the codimension of tangent spaces and note that a type C Schubert variety in $Gr^C(d,2n)$ is not necessarily smooth even if its type A Schubert variety is smooth. Additionally, the condition making such Schubert variety smooth is given.

First, a Schubert variety $X(w)\subset G/P$ is smooth if and only it is smooth at the point $e_{id}$ where $id\in W(G)_P^{\min}$ is the identity in the Weyl group. That is because the singular points of an algebraic variety form a closed subvariety and because the fact that one point in a Schubert variety is singular implies that all points in the same $B$-orbit with this point are singular. Thus, to discuss the smoothness of a Schubert variety, we need to study only its tangent space at $e_{id}$.

For $\underline{i}\in I^{Sp}_{d,2n}$, the defining ideal of $X^A(\underline{i})$ is generated by $\{p_{\underline{j}},\underline{j}\in I_{d,2n}|\underline{j}\not\leq \underline{i}\}$. Thus, the tangent space $T_{e_{id}}(X^A(\underline{i}))$ is determined by Jacobian matrix $$\left.\frac{\partial\{\frac{p_{\underline{j}}}{p_{id}}|\underline{j} \not\leq \underline{i}\}}{\partial\{x_{ij}\}}\right|_{e_{id}},$$  where $id=(1,2,\cdots,d)\in I_{d,2n}$ and $x_{ij}$ are the affine coordinates in the affine neighborhood $A_{id}$ of $e_{id}$.

Now, we want to find the codimension of $T_{e_{id}}(X^C(\underline{i}))$ in $T_{e_{id}}(X^A(\underline{i}))$. Since $X^C(\underline{i})=X^A(\underline{i})\cap Gr^C(d,2n)$,(without confusion we write $E_{\underline{i}}$ instead of $\frac{E_{\underline{i}}}{p_{id}}$) $$T_{e_id}(X^C(\underline{i}))=T_{e_id}(X^A(\underline{i}))\cap (\cap_{\underline{i'}\in I_{d-2,2n},\underline{i'}\subset id} T_{e_{id}}(E_{\underline{i'}})).$$

To determine the codimension, the first problem is to find those $\underline{i'}\subset id$ satisfying $T_{e_{id}}(X^A(\underline{i}))\not\subset T_{e_{id}}(E_{\underline{i'}})$. Note that in every row of the Jacobian matrix, there is at most one nonzero entry (we say that one certain $p_{\underline{j}}$ corresponding to this row is the contributor of $x_{ij}$ corresponding to this column). Thus, $$codim_{T_{e_{id}(X^A(\underline{i}))} } T_{e_{id}}(X^C(\underline{i})) =\#\{\underline{i'}\in I_{d-2,n}|\underline{i'}\subset id, T_{e_{id}}(X^A(\underline{i}))\not\subset T_{e_{id}}(E_{\underline{i'}})\}.$$

In $A_{id}$,$$E_{id\backslash \{s,t\}}=\pm(x_{2n+1-s,t}-x_{2n+1-t,s}+\mbox{deg-2 terms}).$$ Thus, $T_{e_{id}}E_{id\backslash \{s,t\}}$ is the hyperplane $x_{2n+1-s,t}-x_{2n+1-t,s}$. Then $T_{e_{id}}(X^A(\underline{i}))\subset T_{e_{id}}(E_{\underline{i'}})$ if and only if both the contributors of $x_{2n+1-s,t}$ and $x_{2n+1-t,s}$ lie in $\{p_\theta,\theta\not\leq \underline{i}\}$. That is, both of the following statements hold: $$(1,2,\cdots,s-1,s+1,\cdots,t-1,t,t+1,\cdots,d,2n+1-t)\not\leq \underline{i}$$and$$(1,2,\cdots,s-1,s,s+1,\cdots,t-1,t+1,\cdots,d,2n+1-s)\not\leq \underline{i}.$$

\begin{theo}
  Let $\underline{i}\in I^{Sp}_{d,2n}$ and $r=min\{1\leq k\leq d|i_k\geq k+1\}, q=2n+1-i_d\geq r$. Then, the codimension of $T_{e_{id}}(X^C(\underline{i}))$ in $T_{e_{id}}(X^A(\underline{i}))$ is $$\left\{\begin{array}{ll}
                                          \frac{(d-q+1)(d+q-2r)}{2}=\binom{d-r+1}{2}-\binom{q-r}{2} &, \mbox{if } q\leq d\\
                                          0 & ,\mbox{if } q>d
                                        \end{array}\right.$$
\end{theo}
\begin{proof}
  Consider $1\leq s<t\leq d$.
  \begin{itemize}
    \item[case 1.] $s<t<r$. At this time, (the third line is $\underline{i}$) $$\begin{array}{lcrclr}(1,2,\cdots,s-1,&s+1,&\cdots,t-1,&t,&t+1,&\cdots,d,2n+1-t)\\(1,2,\cdots,s-1,&s,&s+1,\cdots,t-1,&t+1,&&\cdots,d,2n+1-s)\\(1,2,\cdots,s-1,&s,&s+1,\cdots,t-1,&t,&\cdots,r-1,&\cdots,2n+1-q)\end{array}$$
$E_{id\backslash\{s,t\}}$ contributes no codimension.
    \item[case 2.] $s<r\leq t$. $$\begin{array}{lcrclr}(1,2,\cdots,s-1,&s+1,&\cdots,t-1,&t,&t+1,&\cdots,d,2n+1-t)\\(1,2,\cdots,s-1,&s,&\cdots,\cdots,t-1,&t+1,&&\cdots,d,2n+1-s)\\(1,2,\cdots,s-1,&s,&\cdots,\cdots,&\cdots&&\cdots,2n+1-q)\end{array}$$
        Note that $2n+1-s>2n+1-r\geq 2n+1-q$, $E_{id\backslash\{s,t\}}$ contributes no codimension.
    \item[case 3.] $r\leq s<t$.
    $$\begin{array}{rcrclr}(1,2,\cdots,s-1,&s+1,&\cdots,t-1,&t,&t+1,&\cdots,d,2n+1-t)\\(1,2,\cdots,s-1,&s,&s+1,\cdots,t-1,&t+1,&&\cdots,d,2n+1-s)\\(1,\cdots,r-1,\cdots,&\cdots,&&\cdots,&&\cdots,2n+1-q)\end{array}.$$ For those $s,t$ satisfying $2n+1-s>2n+1-t>2n+1-q$, $E_{id\backslash\{s,t\}}$ contributes no codimension. Additionally, for $r\leq s,q\leq t$, every equation $E_{id\backslash\{s,t\}}$ contributes codimension 1.
  \end{itemize}
  Thus, for $q> d$, the codimension is zero. For $q\leq d$, the codimension of $T_{e_{id}}(X^C(\underline{i}))$ in $T_{e_{id}}(X^A(\underline{i}))$ is $$\begin{array}{rl}\#\{s<t|r\leq s,q\leq t\}&=(d-q+1)(q-r)+\frac{(d-q+1)(d-q)}{2}\\&=\frac{(d-q+1)(d+q-2r)}{2}.\end{array}$$
\end{proof}

This result also reveals an interesting fact: the codimension of the tangent space $T_{e_{id}}(X^C(\underline{i}))$ in $T_{e_{id}}(X^A(\underline{i}))$ depends only on the indexes $d,q,r$ defined above rather than more specific form of $\underline{i}$.

By means of this theorem and the criterion for smoothness of $X^A(\underline{i})\subset Gr(d,2n),$ $ 1<d<n$ from V. Lakshmibai and J. Weyman\cite[(5.3)]{LAKSHMIBAI1990179}, we can obtain the following result.
\begin{coro}
  For $1<d<n$, $\underline{i}\in I^{Sp}_{d,2n}$ satisfying that $X^A(\underline{i})$ is smooth, then $X^C(\underline{i})$ is smooth if and only if one of the following conditions holds:($q,r$ as defined in Theorem 5.1)
  \begin{itemize}
    \item[(1)]$q> n$ (then trivially $X^C(\underline{i})=X^A(\underline{i})$);
    \item[(2)]$q=r$;
    \item[(3)]$q=r+1$.
  \end{itemize}
\end{coro}
\begin{proof}
  By \cite[(5.3)]{LAKSHMIBAI1990179} the smoothness of $X^A(\underline{i})$ implies $\underline{i}=(1,2,\cdots,r-1,t,t+1,\cdots,t+d-r)$, where $t$ is some integer greater than $r$. Since $\underline{i}\in I^{Sp}_{d,2n}$, there must be $t\geq n$ or $i_d=t+d-r\leq n$.

  If $t\geq n$, by Proposition \ref{number in Ai} we have $\dim X^A(\underline{i})-\dim X^C(\underline{i})=\binom{d-r+1}{2}$. Therefore, $X^C(\underline{i})$ is smooth if and only if $q=r$ or $q=r+1$ by Theorem 5.1. If $i_d\leq n$, the case is trivially $X^C(\underline{i})=X^A(\underline{i})$.
\end{proof}

In general, we have the following criterion immediately.
\begin{coro}
  For $1\leq d\leq n$, $\underline{i}\in I^{Sp}_{d,2n}$, $q,r$ as defined in Theorem 5.1, the Schubert variety $X^C(\underline{i})$ is smooth if and only if
   $$\dim T_{e_{id}}(X^A(\underline{i}))-\dim X^A(\underline{i})=\binom{d-r+1}{2}-\binom{q-r}{2}-\#\{1\leq s<t\leq d|i_s+i_t>2n\}.$$
\end{coro}
\bibliographystyle{plain}
\bibliography{ref}

\begin{thebibliography}{10}

\bibitem{fulton}
{\em Representation theory. A first course.}, volume 129 of {\em Graduate Texts
  in Mathematics}.
\newblock Springer-Verlag, New York, 1991.

\bibitem{Lakshmibai2008}
{\em Symplectic Grassmannian}, pages 55--69.
\newblock Springer Berlin Heidelberg, Berlin, Heidelberg, 2008.

\bibitem{billey_singular_2000}
Sara Billey and V.~Lakshmibai.
\newblock {\em Singular Loci of {Schubert} Varieties}.
\newblock Birkhäuser Boston, Boston, MA, 2000.

\bibitem{code}
Jesús Carrillo-Pacheco and Felipe Zaldivar.
\newblock On {Lagrangian–Grassmannian} codes.
\newblock {\em Designs, Codes and Cryptography}, 60:291--298, 2011.

\bibitem{DeConcini79}
Corrado {De Concini}.
\newblock Symplectic standard tableaux.
\newblock {\em Advances in Mathematics}, 34(1):1--27, 1979.

\bibitem{Hilbert_sym}
Sudhir~R. Ghorpade and K.~N. Raghavan.
\newblock Hilbert functions of points on {Schubert} varieties in the symplectic
  {Grassmannian}.
\newblock {\em Trans. Amer. Math. Soc.}, 358:5401--5423, 2006.

\bibitem{J_HONG}
J.~Hong.
\newblock Classification of smooth {Schubert} varieties in the symplectic
  {Grassmannians}.
\newblock {\em Journal of the Korean Mathematical Society}, 52(5):1109–1122,
  2015.

\bibitem{LAKSHMIBAI1987403GPVII}
V.~Lakshmibai.
\newblock Geometry of ${G/P}$—{VII}—the symplectic group and the involution
  $\sigma$.
\newblock {\em Journal of Algebra}, 108(2):403--434, 1987.

\bibitem{LAK_B}
V.~Lakshmibai.
\newblock Singular loci of {Schubert} varieties for classical groups.
\newblock {\em Bull. Amer. Math. Soc.}, 16:83--90, 1987.

\bibitem{lakshmibai_flag_2018}
V.~Lakshmibai and Justin Brown.
\newblock {\em Flag varieties}, volume~53 of {\em Texts and {Readings} in
  {Mathematics}}.
\newblock Springer Singapore, Singapore, 2018.

\bibitem{LAKSHMIBAI1997332}
V.~Lakshmibai and M.~Song.
\newblock A criterion for smoothness of schubert varieties in
  $\mathrm{Sp}_{2n}/{B}$.
\newblock {\em Journal of Algebra}, 189(2):332--352, 1997.

\bibitem{LAKSHMIBAI1990179}
V~Lakshmibai and J~Weyman.
\newblock Multiplicities of points on a {Schubert} variety in a minuscule
  ${G/P}$.
\newblock {\em Advances in Mathematics}, 84(2):179--208, 1990.

\bibitem{symplectic_multi}
{Minyoung Jeon, Dave Anderson, Takeshi Ikeda, and Ryotaro Kawago}.
\newblock Multiplicities of {Schubert} varieties in the symplectic flag
  variety.
\newblock {\em Sém. Lothar. Combin.}, 82B(95), 2020.

\bibitem{seshadri_introduction_2016}
C.~S. Seshadri.
\newblock {\em Introduction to the {Theory} of {Standard} {Monomials}},
  volume~46 of {\em Texts and {Readings} in {Mathematics}}.
\newblock Springer Singapore, Singapore, 2016.

\bibitem{springer_linear_1998}
T.~A. Springer.
\newblock {\em Linear {Algebraic} {Groups}}.
\newblock Birkhäuser Boston, Boston, MA, 2nd edition, 1998.

\end{thebibliography}
\end{document}